\documentclass{amsart}

\usepackage{amssymb}

\usepackage{url}
\usepackage{hyperref}
\usepackage{thmtools}
\usepackage{tikz-cd}
\usepackage{calc}
\usetikzlibrary{decorations.pathmorphing}
\usepackage{refcount} 

\tikzset{curve/.style={settings={#1},to path={(\tikztostart)
		.. controls ($(\tikztostart)!\pv{pos}!(\tikztotarget)!\pv{height}!270:(\tikztotarget)$)
		and ($(\tikztostart)!1-\pv{pos}!(\tikztotarget)!\pv{height}!270:(\tikztotarget)$)
		.. (\tikztotarget)\tikztonodes}} ,
settings/.code={\tikzset{quiver/.cd,#1}
	\def\pv##1{\pgfkeysvalueof{/tikz/quiver/##1}}} ,
quiver/.cd,pos/.initial=0.35,height/.initial=0}

\tikzset{tail reversed/.code={\pgfsetarrowsstart{tikzcd to}}}
\tikzset{2tail/.code={\pgfsetarrowsstart{Implies[reversed]}}}
\tikzset{2tail reversed/.code={\pgfsetarrowsstart{Implies}}}
\tikzset{no body/.style={/tikz/dash pattern=on 0 off 1mm}}

\declaretheoremstyle[
headfont=\bfseries,              
bodyfont=\itshape,               
headpunct=.,                     
notebraces={}{},                 
headformat=\NAME\ \NOTE          
]{noparensstar}

\usepackage{mathtools}

\newcommand{\Fil}[1]{\textsc{Fil}^*\left(#1\right)}

\newcommand{\adh}[2]{\operatorname{adh}_{#1}\left(#2\right)}

\newtheorem{theorem}{Theorem}[section]
\newtheorem{lemma}{Lemma}[section]
\newtheorem{proposition}{Proposition}[section]
\newtheorem{corollary}{Corollary}[section]

\theoremstyle{definition}

\theoremstyle{remark}
\newtheorem{remark}{Remark}[section]
\usepackage{float}

\numberwithin{equation}{section}
\begin{document}

\title{On convergence structures in graphs}



\author{Paulo Magalhães Junior}
\address{ Departamento de Matemática \\
	Universidade Federal do Piauí\\
	\\ Teresina, PI \\
 64049-550 
}
\curraddr{
}
\email{pjr.mat@ufpi.edu.br} 
\thanks{}

\author{Renan M. Mezabarba}
\address{Departamento de Ciências Exatas, Universidade Estadual de Santa Cruz, Ilhéus, BA, 45662-900}
\curraddr{}
\email{rmmezabarba@uesc.br}
\thanks{}

\author{Rodrigo S. Monteiro}
\address{Instituto de Ciências Matemáticas e de Computação, Universidade de São Paulo, São Carlos, SP, 13566-590}
\curraddr{}
\email{rodrigosm@usp.br}
\thanks{}




\begin{abstract}
A closure operator on a set $X$ is a function $\operatorname{cl}: \wp(X) \to \wp(X)$ satisfying, for all $A, B \subseteq X$, the following properties: extensivity, $A \subseteq \operatorname{cl}(A)$; monotonicity, which states that if $A \subseteq B$ then $\operatorname{cl}(A) \subseteq \operatorname{cl}(B)$; and preservation of unions, $\operatorname{cl}(A \cup B) = \operatorname{cl}(A) \cup \operatorname{cl}(B)$. Every graph $G$ naturally carries such an operator on its vertex set by assigning to each subset $A \subseteq V(G)$ the set $\operatorname{cl}(A) = A \cup N(A)$, where $N(A)$ denotes the vertices adjacent to a vertex in $A$. Since closure operators and pretopological spaces are equivalent notions, this operator induces a canonical convergence structure on $V(G)$. We describe this convergence in terms of nets and relate combinatorial properties of the graph to convergence-theoretic ones.
\end{abstract}


\maketitle

\section*{Introduction}


 In a graph it is possible to consider several natural topologies, depending on the structural aspects one wishes to emphasize. For instance, viewing the graph as a \emph{metric space}, one may define a topology induced by the distance between vertices. Alternatively, interpreting the graph as a \emph{1-simplex} yields a topology based on the geometry of the associated simplicial space. Each of these approaches highlights different features of the global and local structure of the graph. However, from a categorical perspective, an even more natural structure emerges in a graph: a \emph{convergence}.

In the nineteenth and early twentieth centuries, many results in analysis
and topology were formulated in terms of sequences. However, it soon
became clear that in spaces that are not first-countable, sequences are
not sufficient to capture all convergence phenomena. The concept of a
net was introduced by Moore and Smith~\cite{MooreSmith1922}
in 1922 as a generalization of sequences capable of describing
convergence in any topological space. The notion of a filter was
introduced later by Cartan~\cite{cartan1} in 1937. Both concepts provide equivalent ways of describing convergence.
Every net naturally induces a filter generated by its tails,
and a net converges to a point if and only if the induced filter
converges to that point.

 The topology of a space can be completely characterized
by convergence of filters or nets. This observation
leads naturally to the notion of a convergence space, where convergence
of filters (or nets) is taken as the primitive concept.
A \emph{convergence space} is a set equipped with a notion of convergence
for filters or nets, and topological spaces form a special class of
convergence spaces. The modern theory of convergence spaces was
developed by several authors, including Choquet~\cite{Choquet} and
Mynard~\cite{Mynard_convergencia}, among others. A number of books and papers treat the subject,
for instance those by Beattie and Butzmann~\cite{BB}, Binz~\cite{binz},
\v{C}ech~\cite{Cech1966}, Dolecki~\cite{dol}, Dolecki and Mynard~\cite{DM},
Gähler~\cite{gahler}, Nel~\cite{nel}, Preuss~\cite{preuss}, and
Schechter~\cite{schechter}.

In the literature, the definition of convergence spaces is more commonly given via filters, but in this work we adopt the net approach, as in~\cite{nettopology, unknown}. The use of nets in this paper is motivated by convenience: certain arguments can be expressed more directly in terms of eventual behavior along a net. Since filters and nets determine the same convergence structures and can be canonically translated into one another (see~\cite{preuss}), this choice does not affect the underlying theory, but merely streamlines parts of the exposition. Moreover, the formulation in terms of nets often provides more intuition about convergence, as it is typically more natural to think in terms of sequences than in terms of filters.

In this work, we consider the convergence on the vertex set of a graph defined in~\cite{mynardconnec,probconv}. 
Several known results are reproved using nets, and a number of new results concerning the interaction between combinatorial structure of a graph and convergence are established. Whenever a result already exists in the literature, it will be cited in the text; otherwise, it is a new result.  As the interaction between graph theory and convergence theory is still relatively unexplored, we expect that this approach will provide a convenient reference point for further research. Our approach is similar to that of \cite{DovgosheyRovenska2026}, where the authors relate topological properties of ultrametric spaces to combinatorial properties of labeled trees. This relationship is made possible by the categorical equivalence between ultrametric spaces and trees (see \cite{Hughes2004}).

In the literature, as in~\cite{catclosure,ghanim1993compact,milićević2024directedvietorisripscomplexhomotopy,probconv,ShokryYousif2011}, one typically defines a \emph{closure operator} on the vertex set of a graph rather than a convergence structure. These two approaches are ultimately equivalent, since the notion of closure spaces and pretopological spaces, a particular class of convergence spaces, coincide (see~\cite{Cech1966}).    As shown in~\cite{BubenikMilicevic2024,catclosure}, the category of graphs is equivalent to the category of symmetric Alexandroff closure spaces.
To the best of our knowledge, however, this closure operator associated with a graph has received relatively little attention in the literature and is often presented merely as an illustrative example.

The paper is organized as follows. In Section~\ref{sec1}, we review basic concepts of convergence spaces and graph theory. In Section~\ref{sec2}, we introduce the pretopological convergence in a graph and present several characterizations of combinatorial properties in convergence-theoretic properties. Section~\ref{sec3} is devoted to the characterization of connectedness and compactness in a graph viewed as a convergence space, as well as to the exploration of some related consequences. Finally, we conclude with remarks and directions for future work.

\section{Background on graphs and convergence spaces}
\label{sec1}

In this section, we introduce most of the notions and notation used throughout the paper. To keep the exposition largely self-contained, we recall the main definitions concerning graphs and convergence spaces that will be used in the sequel.

 A \emph{filter} $\mathcal{F}$ on a set $X$ is a non-empty family of subsets of $X$ that is closed under finite intersections and upward closed; that is, if $A\in\mathcal{F}$ and $A\subseteq B$, then $B\in\mathcal{F}$. It is called \emph{proper} if $\emptyset\notin\mathcal{F}$. We denote by $\Fil{X}$ the family of proper filters on $X$.

To introduce nets, recall that a directed set is a nonempty set $\mathbb{D}$ equipped with a binary relation $\leq$ that is reflexive and transitive, and such that for every $a,b\in\mathbb{D}$ there exists $c\in\mathbb{D}$ with $a\leq c$ and $b\leq c$. A \emph{net} in a set $X$ is a function $\varphi\colon\mathbb{D}\to X$, where $\mathbb{D}$ is a directed set. We denote by $\textsc{Nets}(X)$ the class of all nets on $X$. When convenient, we write a net as $\langle x_b\rangle_{b\in\mathbb{D}}$, in analogy with sequence notation.

Every net naturally induces a filter. Indeed, if $\varphi\colon\mathbb{D}\to X$ is a net, we define $\varphi^{\uparrow}$ to be the family of subsets of $X$ that contain a tail of $\varphi$, that is, a set of the form $\varphi[a^{\uparrow}]$ for some $a\in\mathbb{D}$, where $a^{\uparrow}=\{b\in\mathbb{D}: b\geq a\}$ and hence $\varphi[a^{\uparrow}]=\{\varphi_b: b\geq a\}$. Conversely, every proper filter is induced by some net (see~\cite{nettopology}).

Let $\langle X,\tau\rangle$ be a topological space. A net $\varphi\in\textsc{Nets}(X)$ is said to $\tau$-\emph{converge} to a point $x\in X$ if every $\tau$-open set $U$ containing $x$ also contains a tail of $\varphi$; equivalently, there exists $a\in\mathbb{D}$ such that $\varphi_b\in U$ for all $b\geq a$. The following result shows that convergence of nets completely characterizes both the topology and continuity.

\begin{proposition}[{\cite[Propositions 1.2.2 and 1.2.3]{monteiro2024algebraictopologyopensets}}]
	\label{opencont}
	Let $X$ and $Y$ be topological spaces.
	\begin{enumerate}
		\item A subset $A\subseteq X$ is open if and only if, for every net
		$\varphi$ in $X$ such that $\varphi\to x$ with $x\in A$, there exists
		$d_0\in\mathrm{dom}(\varphi)$ such that $\varphi_d\in A$ for all
		$d\ge d_0$.
		
		\item A function $f:X\to Y$ is continuous if and only if, for every
		$x\in X$ and every net $\varphi$ in $X$ with $\varphi\to x$, we have
		$f\circ\varphi\to f(x)$.
	\end{enumerate}
\end{proposition}

In the literature, a preconvergence is usually defined as a function $L:\Fil{X}\to\wp(X)$. In this work, however, we adopt a formulation in terms of nets. Thus, a \emph{preconvergence} on $X$ is a class function $L:\textsc{Nets}(X)\to\wp(X)$ such that $L(\varphi)=L(\psi)$ whenever $\varphi^{\uparrow}=\psi^{\uparrow}$. We write $\varphi\to_L x$ instead of $x\in L(\varphi)$ and say that $x$ is an \emph{$L$-limit} of $\varphi$, or that $\varphi$ \emph{$L$-converges} to $x$. The pair $\langle X,L\rangle$ is called a \emph{preconvergence space}. This formulation is not new: in~\cite{schechter}, Schechter translated the usual filter-based definition into the language of nets.

Certain additional properties of a preconvergence are relevant. A preconvergence $L$ on $X$ is said to be \emph{centered} if every constant net converges to its constant value. It is \emph{isotone} if $\psi\to_L x$ whenever $\varphi\to_L x$ and $\psi$ is a subnet of $\varphi$, where, following~\cite{schechter}, a net $\psi$ is called a subnet of $\varphi$ whenever $\varphi^{\uparrow}\subseteq\psi^{\uparrow}$. Finally, $L$ is \emph{stable} if $\rho\to_L x$ whenever $\varphi\to_L x$ and $\psi\to_L x$ and $\rho$ is a mixing of $\varphi$ and $\psi$. Here, given two nets $\varphi,\psi\in\textsc{Nets}(X)$ with the same domain $\mathbb{D}$, a net $\rho\colon\mathbb{D}\to X$ is called a \emph{mixing} of $\varphi$ and $\psi$ if there exists $d'\in\mathbb{D}$ such that $\rho_d\in\{\varphi_d,\psi_d\}$ for all $d\geq d'$.

A preconvergence space $\langle X,L\rangle$ is called a \emph{convergence space} if $L$ is centered and isotone. If, in addition, $L$ is stable, then $\langle X,L\rangle$ is called a \emph{limit space}. A convergence space $\langle X,L\rangle$ is said to be \emph{pretopological} if, for each $x\in X$, there exists a proper filter $\mathcal{F}_x$ converging to $x$ such that $\mathcal{F}_x\subseteq\varphi^{\uparrow}$ whenever $\varphi\in\textsc{Nets}(X)$ satisfies $\varphi\to_L x$. Every topological space is pretopological, and every pretopological space is a limit space; consequently, every topological space is a limit space.

Given a subset $A\subseteq X$ in a preconvergence space $\langle X,L\rangle$, we define its \emph{adherence} by
\[
\mathrm{adh}_L(A)=\{x\in X: \exists \varphi\in\textsc{Nets}(A)\ \text{such that}\ \varphi\to_L x\}.
\]
This yields a map $\mathrm{adh}:\wp(X)\to\wp(X)$, called the \emph{adherence operator}, which satisfies properties analogous to those of the closure operator in topology (see~\cite{dol}).

Motivated by Proposition~\ref{opencont}, we say that a subset $U\subseteq X$ is $L$-\emph{open} if $U\in\varphi^{\uparrow}$ whenever $\varphi\to_L x\in U$. It is straightforward to verify that
\[
\tau_L=\{U\subseteq X: U \text{ is } L\text{-open}\}
\]
is a topology on $X$. The topological space $\langle X,\tau_L\rangle$ is called the \emph{topological modification} of $\langle X,L\rangle$.

In the category of preconvergence spaces, morphisms are again called continuous functions. A function $f\colon X\to Y$ between preconvergence spaces $\langle X,L\rangle$ and $\langle Y,L'\rangle$ is \emph{continuous} if for every $x\in X$ and every net $\varphi\in\textsc{Nets}(X)$ such that $\varphi\to_L x$, one has $f\circ\varphi\to_{L'} f(x)$.

As in the classical topological setting, products and subspaces can be defined naturally. Given a family $\{\langle X_i,L_i\rangle: i\in\mathcal{I}\}$ of preconvergence spaces, the Cartesian product $\prod_{i\in\mathcal{I}}X_i$ is equipped with the \emph{product preconvergence}, in which a net $\varphi$ converges to $\langle x_i\rangle_{i\in\mathcal{I}}$ if and only if $\pi_i\circ\varphi\to_{L_i} x_i$ for every $i\in\mathcal{I}$. Likewise, if $X\subseteq Y$ and $\langle Y,L\rangle$ is a preconvergence space, then $X$ inherits the \emph{subspace preconvergence} $L|_X$, defined by declaring that a net $\psi\in\textsc{Nets}(X)$ converges to $x\in X$ if and only if $\psi\to_L x$ in $Y$.

For graph-theoretic notions and notation we follow~\cite{Diestel2025GraphTheory}. 
Recall that a graph $G$ is a pair $\langle V(G), E(G) \rangle$, where 
$E(G) \subseteq [V(G)]^2$, and $[V(G)]^2$ denotes the family of two-element 
subsets of $V(G)$. The members of $V(G)$ are called \emph{vertices}, and the 
members of $E(G)$ are called \emph{edges}. Throughout, we work with simple 
undirected graphs, so there are no loops, multiple edges, or orientations.

We write $xy\in E(G)$ instead of $\{x,y\}\in E(G)$.
Two vertices $x,y\in V(G)$ are \emph{adjacent}, or \emph{neighbours}, if $xy\in E(G)$. 
If every pair of distinct vertices is adjacent, then $G$ is \emph{complete}.

For $v\in V(G)$, the set of neighbours of $v$ is denoted by $N_G(v)$, or simply $N(v)$ when $G$ is clear from the context. 
We write $N[v]=N(v)\cup\{v\}$. 
More generally, for $U\subseteq V(G)$, the set of neighbours of $U$ in $V(G)\setminus U$ is denoted by $N(U)$, and we define $N[U]=U\cup N(U)$. 
The \emph{degree} of $v$ is $|N(v)|$. 
A graph is \emph{locally finite} if every vertex has finite degree.

A \emph{path} in $G$ is a finite sequence of distinct vertices $v_0v_1\dots v_n$ such that 
$v_iv_{i+1}\in E(G)$ for each $0\le i<n$. A \emph{cycle} is a sequence of vertices 
$v_0v_1\dots v_n$ with $n\ge 3$ such that $v_iv_{i+1}\in E(G)$ for each $0\le i<n$, 
$v_0=v_n$, and $v_0,v_1,\dots,v_{n-1}$ are distinct. The graph $G$ is \emph{connected} 
if it is nonempty and any two vertices can be joined by a path. A \emph{connected component} 
is a maximal connected subgraph. A \emph{tree} is a connected graph that contains no cycles.

Morphisms, products and subspaces in the category of graphs are defined as follows. 
Given graphs $G=\langle V(G),E(G)\rangle$ and $H=\langle V(H),E(H)\rangle$, a function 
$f\colon V(G)\to V(H)$ is a \emph{graph homomorphism} if it preserves adjacency, that is, 
$f(x)f(y)\in E(H)$ whenever $xy\in E(G)$.

The  product of $G$ and $H$ is the graph $G\times H$ defined by
\[
V(G\times H)=V(G)\times V(H)
\]
where two vertices $(u,v)$ and $(u',v')$ are adjacent if and only if $uu'\in E(G)$ and $vv'\in E(H)$.
Finally, if $W\subseteq V(G)$, the \emph{induced subgraph} on $W$, denoted $G[W]$, is the graph
\[
G[W]=\langle W,\{\,uv\in E(G): u,v\in W\,\}\rangle
\]
\section{Graphs as pretopological spaces}

\label{sec2}

We introduce a natural notion of convergence for nets in a graph, based on the
local neighbourhood structure. Let $G$ be a graph and let 
$\varphi\in\textsc{Nets}(V(G))$ be a net in $V(G)$. We say that $\varphi$
\emph{converges} to a vertex $v\in V(G)$ if $N[v]\in\varphi^{\uparrow}$.
This provides a natural notion of convergence, since the closest one can
approach a vertex in a graph is through its neighbours, parametrized by a net.
Moreover, as illustrated in Figure~\ref{ray}, a ray (see the definition in Section~\ref{sec3}) may converge to a vertex
that is adjacent to every vertex of the ray.

This convergence defines a pretopology. Indeed, the filter
$\mathcal{N}_v=\{A\subseteq V(G): N[v]\subseteq A\}$ satisfies
$\mathcal{N}_v\subseteq\varphi^{\uparrow}$ whenever $\varphi\to v$, and it
converges to $v$. It is not hard to see that $\mathrm{adh}(A)=N[A]$ for every
$A\subseteq V(G)$. Thus this pretopology induces a closure operator on $V(G)$,
given by $A\mapsto N[A]$, which coincides with the classical closure structure
associated with a graph in the literature (see~\cite{milićević2024directedvietorisripscomplexhomotopy}).
	
	\begin{figure}[htbp]
		\centering

		\tikzset{every picture/.style={line width=0.75pt}} 
		
		\begin{tikzpicture}[x=0.75pt,y=0.75pt,yscale=-1,xscale=1]
			
			\draw    (281,124) -- (308,124.6) ;
			\draw [shift={(308,124.6)}, rotate = 1.27] [color={rgb, 255:red, 0; green, 0; blue, 0 }  ][fill={rgb, 255:red, 0; green, 0; blue, 0 }  ][line width=0.75]      (0, 0) circle [x radius= 3.35, y radius= 3.35]   ;
			\draw [shift={(281,124)}, rotate = 1.27] [color={rgb, 255:red, 0; green, 0; blue, 0 }  ][fill={rgb, 255:red, 0; green, 0; blue, 0 }  ][line width=0.75]      (0, 0) circle [x radius= 3.35, y radius= 3.35]   ;
			\draw    (308,124.6) -- (335,125.2) ;
			\draw [shift={(335,125.2)}, rotate = 1.27] [color={rgb, 255:red, 0; green, 0; blue, 0 }  ][fill={rgb, 255:red, 0; green, 0; blue, 0 }  ][line width=0.75]      (0, 0) circle [x radius= 3.35, y radius= 3.35]   ;
			\draw [shift={(308,124.6)}, rotate = 1.27] [color={rgb, 255:red, 0; green, 0; blue, 0 }  ][fill={rgb, 255:red, 0; green, 0; blue, 0 }  ][line width=0.75]      (0, 0) circle [x radius= 3.35, y radius= 3.35]   ;
			\draw    (335,125.2) -- (362,125.8) ;
			\draw [shift={(362,125.8)}, rotate = 1.27] [color={rgb, 255:red, 0; green, 0; blue, 0 }  ][fill={rgb, 255:red, 0; green, 0; blue, 0 }  ][line width=0.75]      (0, 0) circle [x radius= 3.35, y radius= 3.35]   ;
			\draw [shift={(335,125.2)}, rotate = 1.27] [color={rgb, 255:red, 0; green, 0; blue, 0 }  ][fill={rgb, 255:red, 0; green, 0; blue, 0 }  ][line width=0.75]      (0, 0) circle [x radius= 3.35, y radius= 3.35]   ;
			\draw    (362,125.8) -- (389,126.4) ;
			\draw [shift={(389,126.4)}, rotate = 1.27] [color={rgb, 255:red, 0; green, 0; blue, 0 }  ][fill={rgb, 255:red, 0; green, 0; blue, 0 }  ][line width=0.75]      (0, 0) circle [x radius= 3.35, y radius= 3.35]   ;
			\draw [shift={(362,125.8)}, rotate = 1.27] [color={rgb, 255:red, 0; green, 0; blue, 0 }  ][fill={rgb, 255:red, 0; green, 0; blue, 0 }  ][line width=0.75]      (0, 0) circle [x radius= 3.35, y radius= 3.35]   ;
			\draw    (389,126.4) -- (416,127) ;
			\draw [shift={(416,127)}, rotate = 1.27] [color={rgb, 255:red, 0; green, 0; blue, 0 }  ][fill={rgb, 255:red, 0; green, 0; blue, 0 }  ][line width=0.75]      (0, 0) circle [x radius= 3.35, y radius= 3.35]   ;
			\draw [shift={(389,126.4)}, rotate = 1.27] [color={rgb, 255:red, 0; green, 0; blue, 0 }  ][fill={rgb, 255:red, 0; green, 0; blue, 0 }  ][line width=0.75]      (0, 0) circle [x radius= 3.35, y radius= 3.35]   ;
			\draw  [dash pattern={on 0.84pt off 2.51pt}]  (426,127) -- (442,126.6) ;
			\draw    (281,124) -- (349,144.6) ;
			\draw [shift={(349,144.6)}, rotate = 16.85] [color={rgb, 255:red, 0; green, 0; blue, 0 }  ][fill={rgb, 255:red, 0; green, 0; blue, 0 }  ][line width=0.75]      (0, 0) circle [x radius= 3.35, y radius= 3.35]   ;
			\draw [shift={(281,124)}, rotate = 16.85] [color={rgb, 255:red, 0; green, 0; blue, 0 }  ][fill={rgb, 255:red, 0; green, 0; blue, 0 }  ][line width=0.75]      (0, 0) circle [x radius= 3.35, y radius= 3.35]   ;
			\draw    (308,124.6) -- (349,144.6) ;
			\draw [shift={(349,144.6)}, rotate = 26] [color={rgb, 255:red, 0; green, 0; blue, 0 }  ][fill={rgb, 255:red, 0; green, 0; blue, 0 }  ][line width=0.75]      (0, 0) circle [x radius= 3.35, y radius= 3.35]   ;
			\draw [shift={(308,124.6)}, rotate = 26] [color={rgb, 255:red, 0; green, 0; blue, 0 }  ][fill={rgb, 255:red, 0; green, 0; blue, 0 }  ][line width=0.75]      (0, 0) circle [x radius= 3.35, y radius= 3.35]   ;
			\draw    (335,125.2) -- (349,144.6) ;
			\draw [shift={(349,144.6)}, rotate = 54.18] [color={rgb, 255:red, 0; green, 0; blue, 0 }  ][fill={rgb, 255:red, 0; green, 0; blue, 0 }  ][line width=0.75]      (0, 0) circle [x radius= 3.35, y radius= 3.35]   ;
			\draw [shift={(335,125.2)}, rotate = 54.18] [color={rgb, 255:red, 0; green, 0; blue, 0 }  ][fill={rgb, 255:red, 0; green, 0; blue, 0 }  ][line width=0.75]      (0, 0) circle [x radius= 3.35, y radius= 3.35]   ;
			\draw    (362,125.8) -- (349,144.6) ;
			\draw [shift={(349,144.6)}, rotate = 124.66] [color={rgb, 255:red, 0; green, 0; blue, 0 }  ][fill={rgb, 255:red, 0; green, 0; blue, 0 }  ][line width=0.75]      (0, 0) circle [x radius= 3.35, y radius= 3.35]   ;
			\draw [shift={(362,125.8)}, rotate = 124.66] [color={rgb, 255:red, 0; green, 0; blue, 0 }  ][fill={rgb, 255:red, 0; green, 0; blue, 0 }  ][line width=0.75]      (0, 0) circle [x radius= 3.35, y radius= 3.35]   ;
			\draw    (389,126.4) -- (349,144.6) ;
			\draw [shift={(349,144.6)}, rotate = 155.53] [color={rgb, 255:red, 0; green, 0; blue, 0 }  ][fill={rgb, 255:red, 0; green, 0; blue, 0 }  ][line width=0.75]      (0, 0) circle [x radius= 3.35, y radius= 3.35]   ;
			\draw [shift={(389,126.4)}, rotate = 155.53] [color={rgb, 255:red, 0; green, 0; blue, 0 }  ][fill={rgb, 255:red, 0; green, 0; blue, 0 }  ][line width=0.75]      (0, 0) circle [x radius= 3.35, y radius= 3.35]   ;
			\draw    (416,127) -- (349,144.6) ;
			\draw [shift={(349,144.6)}, rotate = 165.28] [color={rgb, 255:red, 0; green, 0; blue, 0 }  ][fill={rgb, 255:red, 0; green, 0; blue, 0 }  ][line width=0.75]      (0, 0) circle [x radius= 3.35, y radius= 3.35]   ;
			\draw [shift={(416,127)}, rotate = 165.28] [color={rgb, 255:red, 0; green, 0; blue, 0 }  ][fill={rgb, 255:red, 0; green, 0; blue, 0 }  ][line width=0.75]      (0, 0) circle [x radius= 3.35, y radius= 3.35]   ;
			\draw  [dash pattern={on 0.84pt off 2.51pt}]  (399,141) -- (415,140.6) ;

		\end{tikzpicture}
		
		\caption{A ray converging to a vertex.}
		\label{ray}
	\end{figure}
	In several works, the closure operator is defined for directed graphs (digraphs), that is, graphs in which each edge is assigned an orientation. 
	In this setting, a net converges to a vertex $v \in V(G)$ if some of its tails are eventually contained in 
	\[
	N^+[v] = \{u \in V(G) : (u,v) \in E(G)\}
	\]
	Note that a net that converges when $G$ is regarded as an undirected graph need not converge when $G$ is viewed as a digraph. To the best of our knowledge, convergence in graphs has not previously been formulated in terms of nets in the existing literature.

The convergence structure associated with a graph is, in general, not Hausdorff, since limits of nets need not be unique: a vertex may have several neighbours, and a net may converge to more than one of them. 
This lack of uniqueness is consistent with the combinatorial nature of graphs. Given the definition of convergence adopted here, allowing multiple limit vertices reflects the local adjacency structure and therefore does not constitute a pathological phenomenon.

Before introducing the notion of first countability for convergence spaces, recall that a \emph{basis} of a filter $\mathcal{F}$ is a collection $\mathcal{B}\subseteq \mathcal{F}$ such that
$
\mathcal{F}
=
\{A\subseteq X : \exists B\in\mathcal{B}\text{ such that } B\subseteq A\}
$.
A preconvergence space $\langle X, L\rangle$ is said to be 
\emph{first countable at} $x \in X$ if for every net 
$\varphi \in \textsc{Nets}(X)$ with $\varphi \to_{L} x$, 
there exists a net $\psi \in \textsc{Nets}(X)$ whose induced filter $\psi^\uparrow$ has a countable basis such that 
$\psi \to_{L} x$ and $\varphi$ is a subnet of $\psi$. 
We say that $X$ is \emph{first countable} if it is first countable at each of its points.

The reader can easily verify that the pretopology associated with a graph is first countable and that, for all $x,y\in V(G)$, $\langle y\rangle \to x$ whenever $\langle x\rangle \to y$, where $\langle x\rangle$ denotes the constant net at $x$.
A converse also holds. Let $\langle X,L\rangle$ be a first countable pretopological space such that all neighbourhood filters have a singleton basis and $\langle y\rangle \to x$ whenever $\langle x\rangle \to y$. 
Then $\langle X,L\rangle$ arises from a graph. Indeed, define a graph $G$ by setting $V(G)=X$ and declaring that $xy\in E(G)$ whenever $\langle x\rangle \to y$, and then remove loops and parallel edges.

We begin by showing the relationship between graph homomorphisms and continuous functions.

\begin{theorem}
	\label{cat} Let $G$ and $H$ be graphs.  If $f: G\to H$ is a graph homomorphism, then is continuous as a function between convergence spaces.

\end{theorem}

\begin{proof}
Suppose that $f:G\to H$ is a graph homomorphism. Let $\varphi\in\textsc{Nets}(V(G))$ be a net such that $\varphi\to v$. Then there exists $d\in\mathrm{dom}(\varphi)$ such that $\varphi[d^\uparrow]\subseteq N[v]$. Since $f$ is a graph homomorphism, it follows that $f\circ\varphi[d^\uparrow]\subseteq N[f(v)]$. Hence $f\circ\varphi\to f(v)$, and therefore $f$ is continuous.
\end{proof}

The converse of Theorem~\ref{cat} holds for reflexive graphs, i.e., graphs in which $vv\in E(G)$ for every $v\in V(G)$. The proofs of Propositions~\ref{prod} and \ref{sub} are immediate from the categorical equivalence; nevertheless, we present direct proofs for the sake of completeness.

\begin{proposition}
	\label{prod}
Let $G$ and $H$ be graphs. The convergence associated with the graph product 
$G\times H$ coincides with the product of the convergence structures of $G$ and $H$.
\end{proposition}
\begin{proof}
Denote by $\to_{G\times H}$ the convergence in $G\times H$ as a graph, and by 
$\to_{G\otimes H}$ the convergence in the product of convergence spaces. 
Let $\varphi\in \textsc{Nets}(V(G\times H))$ be a net such that 
$\varphi\to_{G\otimes H} (u,v)$. Then $\pi_G\circ\varphi\to u$ and 
$\pi_H\circ\varphi\to v$. Hence there exists $d\in\mathrm{dom}(\varphi)$ such that
$\pi_G\circ\varphi[d^\uparrow]\subseteq N_G[u]$ and 
$\pi_H\circ\varphi[d^\uparrow]\subseteq N_H[v]$. 
It follows that $\varphi[d^\uparrow]\subseteq N_{G\times H}[(u,v)]$, and therefore 
$\varphi\to_{G\times H}(u,v)$.

Conversely, suppose that $\varphi\to_{G\times H}(u,v)$. Then there exists 
$d\in\mathrm{dom}(\varphi)$ such that 
$\varphi[d^\uparrow]\subseteq N_{G\times H}[(u,v)]$. 
Hence $\pi_G\circ\varphi[d^\uparrow]\subseteq N_G[u]$ and 
$\pi_H\circ\varphi[d^\uparrow]\subseteq N_H[v]$, which implies that 
$\pi_G\circ\varphi\to u$ and $\pi_H\circ\varphi\to v$. 
Therefore $\varphi\to_{G\otimes H}(u,v)$.
\end{proof}

\begin{proposition}
	\label{sub}
Let $G$ be a graph and $A\subseteq V(G)$. The convergence associated with the 
induced subgraph $G[A]$ coincides with the convergence of the subspace $A$.
\end{proposition}
\begin{proof}
Denote by $\to_{G[A]}$ the convergence in $G[A]$ as a graph and by $\to_A$ the convergence in the subspace $A$. 
Let $\varphi\in\textsc{Nets}(A)$ be a net such that $\varphi\to_{G[A]} v\in A$. 
Then $N_{G[A]}[v]\in\varphi^\uparrow$. Since $N_{G[A]}[v]\subseteq N_G[v]$, it follows that $N_G[v]\in\varphi^\uparrow$. 
Hence $\varphi\to v$, and therefore $\varphi\to_A v$.

Conversely, suppose that $\varphi\to_A v\in A$. Then $\varphi\to v$, so there exists $d\in\mathrm{dom}(\varphi)$ such that $\varphi[d^\uparrow]\subseteq N_G[v]$. 
Since $\varphi[d^\uparrow]\subseteq A$ and $G[A]$ is an induced subgraph, it follows that $\varphi[d^\uparrow]\subseteq N_{G[A]}[v]$. 
Therefore $\varphi\to_{G[A]} v$.
\end{proof}

It is well known that not every convergence structure is induced by a topology. A natural question, therefore, is which conditions guarantee the existence of a topology that generates a given convergence.
In the setting of pretopological spaces, this question admits a straightforward answer. We now examine under which conditions the convergence structure associated with a graph is topological.

\begin{theorem}[{\cite[Proposition V.4.4]{DM}}]
	\label{adh}
A pretopological space $\langle X, L \rangle$ is topological if and only if its adherence operator is idempotent, that is,
$\adh{}{A}=\adh{}{\adh{}{A}}$ for every $A\subseteq X$.
\end{theorem}

\begin{theorem}[{\cite[Proposition 3.11]{probconv}}]
The convergence associated with a graph $G$ is topological if and only if $G$ is transitive, 
that is, $xz\in E(G)$ whenever $xy,yz\in E(G)$.
	\label{top}
\end{theorem}
\begin{proof}
Assume that the convergence on $G$ is topological. Let $xy,yz\in E(G)$ and prove that $xz\in E(G)$. 
Notice that $y\in \mathrm{adh}(\{x\})$ and $z\in \mathrm{adh}(\{y\})$. 
By Theorem~\ref{adh}, $\mathrm{adh}(\mathrm{adh}(\{x\}))=\mathrm{adh}(\{x\})$. Since 
$z\in \mathrm{adh}(\mathrm{adh}(\{x\}))$, it follows that $z\in \mathrm{adh}(\{x\})$. 
Hence $xz\in E(G)$, and therefore $G$ is transitive.
	
	Conversely, suppose that $G$ is transitive. By Theorem~\ref{adh}, to show that the convergence on $G$ is topological it suffices to prove that the adherence operator is idempotent. 
	Let $x\in \mathrm{adh}(\mathrm{adh}(A))$. Then there exists $y\in \mathrm{adh}(A)$ such that $xy\in E(G)$. 
	Since $y\in \mathrm{adh}(A)$, there exists $z\in A$ such that $yz\in E(G)$. 
	By transitivity of $G$, it follows that $xz\in E(G)$, and hence $x\in \mathrm{adh}(A)$. 
	Therefore $\mathrm{adh}(A)=\mathrm{adh}(\mathrm{adh}(A))$, showing that the convergence on $G$ is topological.
\end{proof}

\begin{remark}
The proof of Theorem~\ref{top} given in~\cite{probconv} is constructive and formulated for digraphs, where the topology inducing the convergence is explicitly constructed. As a consequence, the argument becomes unnecessarily long. 
Here, we provide a shorter proof based on the adherence operator.
\end{remark}

In general, the graphs we consider are connected and non-complete. One might wonder whether, in most such cases, the associated convergence is topological, so that we would essentially still be working within a topological framework. 
However, this is not the case: connected non-complete graphs are not transitive, and this prevents the convergence from being topological.

We now characterize the notions of locally finite and locally irregular graphs in terms of convergence. For this purpose, we introduce the notions of eventually finite nets and locally irregular graphs.
A net is said to be \emph{eventually finite} if some tail of the net is finite. A graph is said to be \emph{locally irregular} if adjacent vertices always have different degrees.

\begin{theorem}
	\label{localfinite}
Let $G$ be a graph.
\begin{enumerate}
	\item $G$ is locally finite if and only if every convergent net is eventually finite.
	
	\item $G$ is locally irregular if and only if, for every convergent net 
	$\varphi\in\textsc{Nets}(V(G))$ with $\varphi\to v$, there exists 
	$d'\in\mathrm{dom}(\varphi)$ such that for all $d\ge d'$, the vertices 
	$\varphi_d$ and $v$ have different degrees whenever $\varphi_d\neq v$.
\end{enumerate}
\end{theorem}

\begin{proof}
	\text{}
\begin{enumerate}
	\item Suppose that $G$ is locally finite. Let $\varphi\in \textsc{Nets}(V(G))$ be a convergent net, say $\varphi \to v$. By definition, $N[v]\in \varphi^{\uparrow}$. Hence there exists $d\in \mathrm{dom}(\varphi)$ such that $\varphi[d^\uparrow]\subseteq N[v]$. Since $G$ is locally finite, $N[v]$ is finite, and therefore $\varphi[d^\uparrow]$ is finite. This proves that $\varphi$ is eventually finite.
	
	Conversely, suppose that $G$ is not locally finite. Let $v\in V(G)$ be a vertex of infinite degree. Then there exists an injective sequence $\langle v_n\rangle_{n}$ of neighbours of $v$. This sequence converges to $v$, but it is not eventually finite.
	
	\item Suppose that $G$ is locally irregular. Let $\varphi\to v$ be a convergent net. Then there exists $d'\in \mathrm{dom}(\varphi)$ such that $\varphi[d'^\uparrow]\subseteq N[v]$. Hence $\varphi_d$ and $v$ are adjacent for every $d\ge d'$ whenever $\varphi_d\neq v$. Since $G$ is locally irregular, $\varphi_d$ and $v$ have different degrees.
	
	Conversely, let $x$ and $y$ be adjacent vertices. We must show that they have different degrees. Notice that $\langle x\rangle\to y$. Since the image of this net is the single vertex $x$, which is distinct from $y$, the hypothesis implies that $x$ and $y$ have different degrees. Therefore $G$ is locally irregular.
\end{enumerate}
\end{proof}

We now turn to bipartite graphs. Our aim is to characterize the bipartite property in terms of convergence. Recall that a graph is called \emph{bipartite} if its vertex set can be partitioned into two classes in such a way that no edge joins two vertices in the same class.

Notice that the partition $V(G)=A\cup B$ induces a natural separation of the vertex set into two independent classes. The following theorem shows that this combinatorial condition can be equivalently described in terms of the behavior of convergent nets: convergence respects the bipartition in the sense that every net converging to a vertex is eventually contained in the opposite class of the limit point.

\begin{theorem}
	\label{bipartite}
	A graph $G$ is bipartite if and only if there exist disjoint subsets 
	$A,B\subseteq V(G)$ with $V(G)=A\cup B$ such that, for every net 
	$\varphi\in\textsc{Nets}(V(G))$ with $\varphi\to x$, we have 
	$B\cup\{x\}\in\varphi^\uparrow$ whenever $x\in A$, and 
	$A\cup\{x\}\in\varphi^\uparrow$ whenever $x\in B$.
\end{theorem}

\begin{proof}
	Let $A$ and $B$ be the two sides of the bipartite graph $G$. 
	Consider a net $\varphi\in\textsc{Nets}(V(G))$ such that $\varphi\to v$. 
	Then there exists $d\in\mathrm{dom}(\varphi)$ such that 
	$\varphi[d^\uparrow]\subseteq N[v]$. 
	If $v\in A$, then $N[v]\subseteq B\cup\{v\}$, and if $v\in B$, then 
	$N[v]\subseteq A\cup\{v\}$.
	
	Conversely, let $x\in A$ and suppose that $xy\in E(G)$. 
	We show that $y\in B$. Notice that $\langle y\rangle\to x$. 
	Since $x\in A$, it follows that $B\cup\{x\}\in\langle y\rangle^\uparrow$. 
	As $x\neq y$, we conclude that $y\in B$. 
	The case $x\in B$ is analogous. Hence $G$ is bipartite.
\end{proof}

\begin{theorem}
	A graph $G$ is complete if and only if every net converges to all vertices of $G$.
\end{theorem}

\begin{proof}
	Let $G$ be a complete graph. Notice that the image of any net $\varphi\in\textsc{Nets}(V(G))$ is contained in $N[v]$ for every vertex $v \in V(G)$, because $G$ is complete. This proves that $\varphi \to v$ for all $v \in V(G)$.

	Conversely, in the case that all nets converges to all vertices, for two vertices $u, v \in V(G)$, by hypothesis we have $\langle u \rangle \to v$. This implies that $uv \in E(G)$. Therefore, $G$ is complete.
\end{proof}
Recall that two topologies on a set can be compared by inclusion. 
A topology $\tau'$ on a set $X$ is said to be \emph{finer} than a topology $\tau$ on the same set if $\tau \subseteq \tau'$, that is, if every open set in $\tau$ is also open in $\tau'$.
In this case, every net that converges with respect to $\tau'$ also converges with respect to $\tau$. 
This property will serve as the key idea for defining an order between convergence structures.

Given preconvergences $L$ and $L'$ on a set $X$, we say that $L'$ is \emph{finer} (or \emph{stronger}) than $L$ if
\[
\varphi \to_{L'} x \;\Rightarrow\; \varphi \to_{L} x
\]
for every $\varphi \in \textsc{Nets}(X)$ and every $x \in X$. In this case, we write $L \le L'$.

This order on preconvergences is closely related to the notion of spanning subgraphs, that is, subgraphs that contain all the vertices of the graph.

\begin{theorem}
	Let $G$ and $H$ be graphs with the same vertex set, the following are equivalent
	
	\begin{enumerate}
		\item $H$ is a spanning subgraph of $G$.
		\item The convergence on $H$ is stronger than the convergence on $G$.
	\end{enumerate}
\end{theorem}
\begin{proof}
Let $H$ and $G$ be graphs such that $V(H)=V(G)$. Denote by $\to_H$ and $\to_G$ the convergences in $H$ and $G$, respectively. Suppose that $H$ is a subgraph of $G$, that is, a spanning subgraph. We prove that $\to_G \leq \to_H$.

Let $\varphi \in \textsc{Nets}(V(H))$ be a net such that $\varphi \to_H v$. Then $N_H[v] \in \varphi^{\uparrow}$. Since $H$ is a subgraph of $G$, we have $N_H[v] \subseteq N_G[v]$. Hence $N_G[v] \in \varphi^{\uparrow}$, and therefore $\varphi \to_G v$. This proves that $\to_G \leq \to_H$.

Conversely, suppose that $\to_G \leq \to_H$. We must show that every edge of $H$ is also an edge of $G$. Let $xy \in E(H)$. Then $\langle x \rangle \to_H y$. Since $\to_G \leq \to_H$, it follows that $\langle x \rangle \to_G y$. Hence $xy \in E(G)$, and therefore $H$ is a subgraph of $G$.
\end{proof}

\section{Connectedness and Compactness: Characterizations and Consequences}

Although convergence in a graph is not always topological, it admits an associated topological modification. In the following results, we study the structure of the open and closed sets in this topology, which will allow us to relate the connectedness of the graph with connectedness in the associated convergence space.

\begin{proposition}
	\label{clopen}
	Let $G$ be a graph and let $U,F\subseteq V(G)$.
	
	\begin{enumerate}
		\item A subset $U$ is open if and only if $N[v]\subseteq U$ for every $v\in U$.
		
		\item A subset $F$ is closed if and only if it is a union of connected components of $G$ as a graph.
	\end{enumerate}
\end{proposition}
\begin{proof}
	\text{}
	\begin{enumerate}
		\item Suppose that $U$ is open. Let $v\in U$ and $u\in N[v]$. Then $\langle u\rangle \to v$. Since $U$ is open, some tail of the net $\langle u\rangle$ is contained in $U$. Hence $u\in U$, and therefore $N[v]\subseteq U$.
		
		Conversely, suppose that $N[u]\subseteq U$ for every $u\in U$. Let $\varphi\in\textsc{Nets}(V(G))$ be such that $\varphi \to u$ with $u\in U$. Then $N[u]\in\varphi^{\uparrow}$. Since $N[u]\subseteq U$ and $\varphi^{\uparrow}$ is a filter, it follows that $U\in\varphi^{\uparrow}$. Hence $U$ is open.
		
		\item Suppose that $F$ is closed. Then the connected component of $G$ containing any $v\in F$ is contained in $F$. Hence $F$ is a union of connected components of $G$.
		
		Conversely, suppose that $F$ is a union of connected components of $G$. Let $v\in \mathrm{adh}(F)$. Then there exists $u\in F$ such that $uv\in E(G)$. Since $F$ is a union of connected components, there exists a connected component $C$ such that $u\in V(C)\subseteq F$. Hence $v\in V(C)\subseteq F$. Therefore $F$ is closed.
	\end{enumerate}
\end{proof}

To the best of our knowledge, only item $(1)$ of Proposition~\ref{clopen} appears in the literature; it is proved in~\cite{probconv} using filters, albeit with a slightly different argument. The Proposition \ref{clopen} implies that every open set is closed and conversely, and that the connected components form a basis for the topology. Consequently, we obtain the following corollary.

\begin{corollary}
	\label{topmod}
	The topological modification associated with the graph convergence turns out to be an Alexandroff space, in which open sets are precisely unions of connected components.
\end{corollary}

Our goal now is to prove that connectedness in a graph implies path-connectedness when the graph is regarded as a convergence space. To achieve this, we first establish a version of the pasting lemma for limit spaces. It is worth noting that, as shown in~\cite{preuss}, this lemma does not hold for every type of preconvergence space. This result ensures, in particular, that if there is a path from $x$ to $y$ and a path from $y$ to $z$, then these can be concatened to obtain a path from $x$ to $z$.
A path in a preconvergence space $X$ is a continuous function $\gamma\colon[0,1]\to X$, where $[0,1]$ is endowed with its usual topology. If $\gamma(0)=\gamma(1)$ we say that $\gamma$ is a loop. The definition of a path-connected preconvergence space is the same as in the topological setting.

We say that a family $\mathcal{C}$ of subsets of a preconvergence space $X$ is \emph{locally finite} if for every $x \in X$ there exists an open set $U \subseteq X$ such that $x \in U$ and the set
$
\{\, V \in \mathcal{C} : V \cap U \neq \emptyset \,\}
$
is finite.

\begin{lemma}[{\cite[Lemma 3.2]{nettopology}}]
	\label{pastinglemma}
	Let $X$ and $Y$ be limit spaces, and let $f\colon X\to Y$ be a function. If $\mathcal{C}$ is a locally finite cover of $X$ by closed sets such that the restriction $f|_C$ is continuous for every $C \in \mathcal{C}$, then $f$ is continuous.
\end{lemma}
\begin{restatable}{theorem}{path}
	\label{path}
	Every connected graph $G$ is path-connected as a convergence space.
\end{restatable}
\begin{proof}
	By Pasting Lemma~\ref{pastinglemma}, since $G$ is connected, it is enough to show that any two adjacents vertices are joined by a path in the convergence space. Let $xy\in E(G)$ be an edge of $G$, and consider the function $\gamma:[0,1]\to V(G)$ defined by
	
	\[
	\gamma(t)=
	\begin{cases}
		x & \text{if } 0\leq t\le \tfrac{1}{2},\\[6pt]
		y & \text{if } \tfrac{1}{2}< t\le 1.
	\end{cases}
	\]

	Observe that $\gamma$ is continuous because $x$ and $y$ are adjacent. Hence $G$ is a path-connected convergence space.
\end{proof}

	\begin{remark}
		\label{repath}
		The proof of Theorem \ref{path} is independent of the Pasting Lemma \ref{pastinglemma}; we use it only to simplify the argument. In fact, a path in a graph can naturally be viewed as a path in the associated convergence space.
		
		Let $v_0v_1\dots v_n$ be a path in $G$. Divide the interval $[0,1]$ into $n$ equal subintervals and define a function $\gamma:[0,1]\to V(G)$, where $[0,1]$ is endowed with its usual topology, by
		\[
		\gamma(x)=
		\begin{cases}
			v_0 & \text{if } x=0,\\
			v_k & \text{if } \frac{k-1}{n} < x \le \frac{k}{n}, \quad k=1,\dots,n.
		\end{cases}
		\]
		Then $\gamma$ is continuous, $\gamma(0)=v_0$ and $\gamma(1)=v_n$.
	\end{remark}

A preconvergence space $\langle X, L \rangle$ is said to be connected if every continuous function $f : X \to \{0,1\}$ is constant, where $\{0,1\}$ is endowed with the discrete topology. Notice that this is equivalent to requiring that the topological modification $\langle X, \tau_L \rangle$ be a connected topological space.

In \cite{NeumannLaraWilson1995}, Victor Neumann-Lara and Richard Wilson define a compatible topology on the vertex set of a graph as a topology such that every induced subgraph of $G$ is connected if and only if its vertex set is topologically connected. Their paper concludes with the following question: \textit{Which graphs have compatible compact topologies?} That is, graphs for which every compatible topology is compact.

In what follows, we prove that the topological modification of a graph is a compatible topology. Furthermore, we raise the following question: which combinatorial properties must a graph satisfy so that any convergence whose topological modification is a compatible topology is a compact convergence (in the sense to be defined later)?

\begin{theorem}
	\label{connected}
	For every graph $G$, the following are equivalent:
	\begin{enumerate}
		\item $G$ is connected as a graph;
		\item $G$ is connected as a convergence space.
	\end{enumerate}
\end{theorem}

\begin{proof}
	The proof basically follows from Proposition \ref{clopen}, but we present an alternative proof. 
	
	First, let us consider $G$ as a connected graph and show that it is also connected as a convergence space. Let $f:V(G)\to\{0,1\}$ be a continuous function. We must prove that $f$ is constant. For this, since $G$ is connected, it is enough to prove that $f(x)=f(y)$ for $xy\in E(G)$. We have $\langle x\rangle \to x$ and $\langle x\rangle \to y$. Then $f\circ \langle x\rangle \to f(x)$ and $f\circ \langle x\rangle \to f(y)$. Then $f(x)=f(y)$. This proves that $G$ is connected as convergence space.
	
	Conversely, suppose that $G$ is a disconnected graph. There are vertices $x,y\in V(G)$ that are not joined by a path in $G$. Consider the function $f: V(G)\to \{0,1\}$ such that $f(z)=0$ if $z=x$ or there is path between $z$ and $x$,  and $f(z)=1$  if $z=y$ or there is path between $z$ and $y$, and $f(z)=0$ otherwise. We prove that $f$ is continuous. Given a net $\varphi\in\textsc{Nets}(V(G)) $ such that $\varphi \to v$. There is $d\in \hbox{dom}(\varphi)$ such that $\varphi [d^\uparrow]\subseteq N[v]$. The neighbors of $z$ share the same properties as $z$ concerning the existence of paths between $x$ and $y$. Then $f\circ \varphi[d^\uparrow]\subseteq \{f(z)\}$. This prove that $f\circ\varphi \to f(v)$. Then $f$ is continuous. Since $f$ is not constant, $G$ can not be connected as convergence space.
\end{proof}
\begin{remark}
	\text{}
	\begin{enumerate}
	
		\item In~\cite{mynardconnec}, Herrejón and Mynard present a version of Theorem \ref{connected} for finite digraphs using filters, in which connectedness as a convergence space is equivalent to connectedness of the underlying graph when edge orientations are ignored. In their work, finite digraphs are also used to construct examples involving enclosing sets, $T$-subspaces, and sandwiched sets in convergence spaces.
		
		\item Recall that a graph $G$ is said to be \emph{2-connected} if it is connected and remains connected after the removal of any single vertex. By Theorem \ref{connected}, $G$ is $2$-connected if and only if it is connected and, for every $v \in V(G)$, every continuous function $V(G)\setminus\{v\} \to \{0,1\}$ admits a continuous extension to $V(G)$.
	\end{enumerate}
\end{remark}

\begin{corollary}
Let $T$ be a graph. If $T$ is path-connected as a convergence space and has no loops, then is a tree.
\end{corollary}

\begin{proof}
	
		Suppose that $T$ is path-connected as a convergence space and has no loops. 
		Since a path-connected convergence space has a path-connected topological modification, it follows that this topological space is connected. 
		By Theorem~\ref{connected}, the graph $T$ is therefore connected. 
		Moreover, $T$ has no cycles, because any cycle in the graph would correspond to a loop in the convergence space by Remark~\ref{repath}. 
		Hence, $T$ is connected and acyclic, that is, a tree.
\end{proof}

\label{sec3}

We now establish a connection between compactness as a convergence space and the existence of finite dominating sets in a graph. Recall that a topological space is compact if every open cover has a finite subcover, or equivalently, if every net has a convergent subnet. With this in mind, a preconvergence space $X$ is said to be \emph{compact} if every net has a convergent subnet.

To obtain a definition analogous to open covers in the topological case, we need to generalize the concept of open covers. In our setting, these are the convergence systems. Let $\langle X,L\rangle$ be a preconvergence space. A family $\mathcal{C}$ of subsets of $X$ is a \emph{local convergence system} at $x\in X$ if, for every net $\varphi\in\textsc{Nets}(X)$ such that $\varphi\to_L x$, there exists $C\in\mathcal{C}$ with $C\in\varphi^{\uparrow}$. We say that $\mathcal{C}$ is a \emph{convergence system} if it is a local convergence system for each $x\in X$. The usual terminology is \emph{covering system}, as in~\cite{BB}.

\begin{theorem}[{\cite[Proposition 1.4.15]{BB}}]
			A convergence space $\langle X,L\rangle$ is compact if and only if every convergence system has a finite subcover.
	\end{theorem}

\begin{proposition}[{\cite[Lemma 7.4]{Jech}}]
	Let $\mathcal{F}$ be a proper filter on $X$. The following statements are equivalent:
	\begin{enumerate}
		\item $\mathcal{F}$ is an ultrafilter on $X$, i.e, a maximal filter with respect to inclusion;
		\item For every $A \subseteq X$, either $A \in \mathcal{F}$ or $X \setminus A \in \mathcal{F}$.

	\end{enumerate}
\end{proposition}

A subset $D \subseteq V(G)$ is called a \emph{dominating set} of $G$ if every vertex $v \in V(G) \setminus D$ is adjacent to at least one vertex $u \in D$ (see Figure \ref{exm3}). Equivalently, for all $v \in V(G)$, either $v \in D$ or there exists $u \in D$ such that $uv \in E(G)$. In convergence context, this is a finite dense subset.

\begin{theorem}
	\label{teocomp2}
	
		For every graph $G$, the following are equivalent:
	\begin{enumerate}
		\item $G$ is compact as convergence space;
		\item  $G$ has a finite dominating set.
	\end{enumerate}

\end{theorem}

\begin{proof}
	Suppose that $G$ is compact as convergence space. Notice that, by the definition of convergence on $V(G)$, the family $\mathcal{N}=\{N[v]:v\in V(G)\}$ is a convergence system. Since $G$ is compact, $\mathcal{N}$ has a finite subcover. Then there are $v_0,\cdots,v_n\in V(G)$ such that $V(G)=\bigcup_{0\leq i\leq n} N[v_i]$. It follows that $D=\{v_0,v_1,\cdots,v_n\}$ is a finite dominating subset.
	
	Conversely, suppose that $G$ has a finite dominating subset $D$. 
	It is well known that, by Zorn's Lemma, every filter is contained in an ultrafilter 
	(see \cite[Lemma 7.5]{Jech}). Then it is sufficient to show that every net $\varphi\in\textsc{Nets}(V(G))$  whose induced filter $\varphi^\uparrow$ is an ultrafilter converges. Notice that $V(G)=N[D]\in\varphi^\uparrow$. Suppose, for the sake of contradiction, that $\varphi$ does not converge. For every $v\in V(G)$ we have $N[v]\notin \varphi^\uparrow$. Since $\varphi^\uparrow$ is an ultrafilter, it follows that $V(G)\setminus N[v]\in \varphi^\uparrow$.  In this case, we have $ \bigcap_{v\in D} (V(G)\setminus N[v])=V(G)\setminus N[D]=\emptyset\in\varphi^\uparrow$, since $D$ is finite. This contradiction tells us that $\varphi$ converges. Therefore, $G$ is compact as convergence space.
\end{proof}

From this point onward, the term \emph{compact graph} will always mean a graph that is compact as a convergence space.

\begin{figure}[H]
	\centering

	\tikzset{every picture/.style={line width=0.75pt}} 
	
	\begin{tikzpicture}[x=0.75pt,y=0.75pt,yscale=-1,xscale=1]
		
		\draw [color={rgb, 255:red, 0; green, 0; blue, 0 }  ,draw opacity=1 ]   (301,123.6) -- (321,123.6) ;
		\draw [shift={(321,123.6)}, rotate = 0] [color={rgb, 255:red, 0; green, 0; blue, 0 }  ,draw opacity=1 ][fill={rgb, 255:red, 0; green, 0; blue, 0 }  ,fill opacity=1 ][line width=0.75]      (0, 0) circle [x radius= 3.35, y radius= 3.35]   ;
		\draw [shift={(301,123.6)}, rotate = 0] [color={rgb, 255:red, 0; green, 0; blue, 0 }  ,draw opacity=1 ][fill={rgb, 255:red, 0; green, 0; blue, 0 }  ,fill opacity=1 ][line width=0.75]      (0, 0) circle [x radius= 3.35, y radius= 3.35]   ;
		\draw [color={rgb, 255:red, 0; green, 0; blue, 0 }  ,draw opacity=1 ]   (341,123.6) -- (361,123.6) ;
		\draw [shift={(361,123.6)}, rotate = 0] [color={rgb, 255:red, 0; green, 0; blue, 0 }  ,draw opacity=1 ][fill={rgb, 255:red, 0; green, 0; blue, 0 }  ,fill opacity=1 ][line width=0.75]      (0, 0) circle [x radius= 3.35, y radius= 3.35]   ;
		\draw [shift={(341,123.6)}, rotate = 0] [color={rgb, 255:red, 0; green, 0; blue, 0 }  ,draw opacity=1 ][fill={rgb, 255:red, 0; green, 0; blue, 0 }  ,fill opacity=1 ][line width=0.75]      (0, 0) circle [x radius= 3.35, y radius= 3.35]   ;
		\draw [color={rgb, 255:red, 0; green, 0; blue, 0 }  ,draw opacity=1 ]   (381,123.6) -- (401,123.6) ;
		\draw [shift={(401,123.6)}, rotate = 0] [color={rgb, 255:red, 0; green, 0; blue, 0 }  ,draw opacity=1 ][fill={rgb, 255:red, 0; green, 0; blue, 0 }  ,fill opacity=1 ][line width=0.75]      (0, 0) circle [x radius= 3.35, y radius= 3.35]   ;
		\draw [shift={(381,123.6)}, rotate = 0] [color={rgb, 255:red, 0; green, 0; blue, 0 }  ,draw opacity=1 ][fill={rgb, 255:red, 0; green, 0; blue, 0 }  ,fill opacity=1 ][line width=0.75]      (0, 0) circle [x radius= 3.35, y radius= 3.35]   ;
		\draw [color={rgb, 255:red, 0; green, 0; blue, 0 }  ,draw opacity=1 ]   (301,144.6) -- (321,144.6) ;
		\draw [shift={(321,144.6)}, rotate = 0] [color={rgb, 255:red, 0; green, 0; blue, 0 }  ,draw opacity=1 ][fill={rgb, 255:red, 0; green, 0; blue, 0 }  ,fill opacity=1 ][line width=0.75]      (0, 0) circle [x radius= 3.35, y radius= 3.35]   ;
		\draw [shift={(301,144.6)}, rotate = 0] [color={rgb, 255:red, 0; green, 0; blue, 0 }  ,draw opacity=1 ][fill={rgb, 255:red, 0; green, 0; blue, 0 }  ,fill opacity=1 ][line width=0.75]      (0, 0) circle [x radius= 3.35, y radius= 3.35]   ;
		\draw [color={rgb, 255:red, 0; green, 0; blue, 0 }  ,draw opacity=1 ]   (341,144.6) -- (361,144.6) ;
		\draw [shift={(361,144.6)}, rotate = 0] [color={rgb, 255:red, 0; green, 0; blue, 0 }  ,draw opacity=1 ][fill={rgb, 255:red, 0; green, 0; blue, 0 }  ,fill opacity=1 ][line width=0.75]      (0, 0) circle [x radius= 3.35, y radius= 3.35]   ;
		\draw [shift={(341,144.6)}, rotate = 0] [color={rgb, 255:red, 0; green, 0; blue, 0 }  ,draw opacity=1 ][fill={rgb, 255:red, 0; green, 0; blue, 0 }  ,fill opacity=1 ][line width=0.75]      (0, 0) circle [x radius= 3.35, y radius= 3.35]   ;
		\draw [color={rgb, 255:red, 0; green, 0; blue, 0 }  ,draw opacity=1 ]   (381,144.6) -- (401,144.6) ;
		\draw [shift={(401,144.6)}, rotate = 0] [color={rgb, 255:red, 0; green, 0; blue, 0 }  ,draw opacity=1 ][fill={rgb, 255:red, 0; green, 0; blue, 0 }  ,fill opacity=1 ][line width=0.75]      (0, 0) circle [x radius= 3.35, y radius= 3.35]   ;
		\draw [shift={(381,144.6)}, rotate = 0] [color={rgb, 255:red, 0; green, 0; blue, 0 }  ,draw opacity=1 ][fill={rgb, 255:red, 0; green, 0; blue, 0 }  ,fill opacity=1 ][line width=0.75]      (0, 0) circle [x radius= 3.35, y radius= 3.35]   ;
		\draw [color={rgb, 255:red, 0; green, 0; blue, 0 }  ,draw opacity=1 ]   (301,123.6) -- (301,144.6) ;
		\draw [shift={(301,144.6)}, rotate = 90] [color={rgb, 255:red, 0; green, 0; blue, 0 }  ,draw opacity=1 ][fill={rgb, 255:red, 0; green, 0; blue, 0 }  ,fill opacity=1 ][line width=0.75]      (0, 0) circle [x radius= 3.35, y radius= 3.35]   ;
		\draw [shift={(301,123.6)}, rotate = 90] [color={rgb, 255:red, 0; green, 0; blue, 0 }  ,draw opacity=1 ][fill={rgb, 255:red, 0; green, 0; blue, 0 }  ,fill opacity=1 ][line width=0.75]      (0, 0) circle [x radius= 3.35, y radius= 3.35]   ;
		\draw [color={rgb, 255:red, 0; green, 0; blue, 0 }  ,draw opacity=1 ]   (321,123.6) -- (321,144.6) ;
		\draw [shift={(321,144.6)}, rotate = 90] [color={rgb, 255:red, 0; green, 0; blue, 0 }  ,draw opacity=1 ][fill={rgb, 255:red, 0; green, 0; blue, 0 }  ,fill opacity=1 ][line width=0.75]      (0, 0) circle [x radius= 3.35, y radius= 3.35]   ;
		\draw [shift={(321,123.6)}, rotate = 90] [color={rgb, 255:red, 0; green, 0; blue, 0 }  ,draw opacity=1 ][fill={rgb, 255:red, 0; green, 0; blue, 0 }  ,fill opacity=1 ][line width=0.75]      (0, 0) circle [x radius= 3.35, y radius= 3.35]   ;
		\draw    (341,123.6) -- (341,144.6) ;
		\draw [shift={(341,144.6)}, rotate = 90] [color={rgb, 255:red, 0; green, 0; blue, 0 }  ][fill={rgb, 255:red, 0; green, 0; blue, 0 }  ][line width=0.75]      (0, 0) circle [x radius= 3.35, y radius= 3.35]   ;
		\draw [shift={(341,123.6)}, rotate = 90] [color={rgb, 255:red, 0; green, 0; blue, 0 }  ][fill={rgb, 255:red, 0; green, 0; blue, 0 }  ][line width=0.75]      (0, 0) circle [x radius= 3.35, y radius= 3.35]   ;
		\draw    (361,123.6) -- (361,144.6) ;
		\draw [shift={(361,144.6)}, rotate = 90] [color={rgb, 255:red, 0; green, 0; blue, 0 }  ][fill={rgb, 255:red, 0; green, 0; blue, 0 }  ][line width=0.75]      (0, 0) circle [x radius= 3.35, y radius= 3.35]   ;
		\draw [shift={(361,123.6)}, rotate = 90] [color={rgb, 255:red, 0; green, 0; blue, 0 }  ][fill={rgb, 255:red, 0; green, 0; blue, 0 }  ][line width=0.75]      (0, 0) circle [x radius= 3.35, y radius= 3.35]   ;
		\draw    (381,123.6) -- (381,144.6) ;
		\draw [shift={(381,144.6)}, rotate = 90] [color={rgb, 255:red, 0; green, 0; blue, 0 }  ][fill={rgb, 255:red, 0; green, 0; blue, 0 }  ][line width=0.75]      (0, 0) circle [x radius= 3.35, y radius= 3.35]   ;
		\draw [shift={(381,123.6)}, rotate = 90] [color={rgb, 255:red, 0; green, 0; blue, 0 }  ][fill={rgb, 255:red, 0; green, 0; blue, 0 }  ][line width=0.75]      (0, 0) circle [x radius= 3.35, y radius= 3.35]   ;
		\draw    (401,123.6) -- (401,144.6) ;
		\draw [shift={(401,144.6)}, rotate = 90] [color={rgb, 255:red, 0; green, 0; blue, 0 }  ][fill={rgb, 255:red, 0; green, 0; blue, 0 }  ][line width=0.75]      (0, 0) circle [x radius= 3.35, y radius= 3.35]   ;
		\draw [shift={(401,123.6)}, rotate = 90] [color={rgb, 255:red, 0; green, 0; blue, 0 }  ][fill={rgb, 255:red, 0; green, 0; blue, 0 }  ][line width=0.75]      (0, 0) circle [x radius= 3.35, y radius= 3.35]   ;
		\draw [color={rgb, 255:red, 0; green, 0; blue, 0 }  ,draw opacity=1 ]   (321,123.6) -- (341,123.6) ;
		\draw [shift={(341,123.6)}, rotate = 0] [color={rgb, 255:red, 0; green, 0; blue, 0 }  ,draw opacity=1 ][fill={rgb, 255:red, 0; green, 0; blue, 0 }  ,fill opacity=1 ][line width=0.75]      (0, 0) circle [x radius= 3.35, y radius= 3.35]   ;
		\draw [shift={(321,123.6)}, rotate = 0] [color={rgb, 255:red, 0; green, 0; blue, 0 }  ,draw opacity=1 ][fill={rgb, 255:red, 0; green, 0; blue, 0 }  ,fill opacity=1 ][line width=0.75]      (0, 0) circle [x radius= 3.35, y radius= 3.35]   ;
		\draw [color={rgb, 255:red, 0; green, 0; blue, 0 }  ,draw opacity=1 ]   (361,123.6) -- (381,123.6) ;
		\draw [shift={(381,123.6)}, rotate = 0] [color={rgb, 255:red, 0; green, 0; blue, 0 }  ,draw opacity=1 ][fill={rgb, 255:red, 0; green, 0; blue, 0 }  ,fill opacity=1 ][line width=0.75]      (0, 0) circle [x radius= 3.35, y radius= 3.35]   ;
		\draw [shift={(361,123.6)}, rotate = 0] [color={rgb, 255:red, 0; green, 0; blue, 0 }  ,draw opacity=1 ][fill={rgb, 255:red, 0; green, 0; blue, 0 }  ,fill opacity=1 ][line width=0.75]      (0, 0) circle [x radius= 3.35, y radius= 3.35]   ;
		\draw [color={rgb, 255:red, 0; green, 0; blue, 0 }  ,draw opacity=1 ]   (321,144.6) -- (341,144.6) ;
		\draw [shift={(341,144.6)}, rotate = 0] [color={rgb, 255:red, 0; green, 0; blue, 0 }  ,draw opacity=1 ][fill={rgb, 255:red, 0; green, 0; blue, 0 }  ,fill opacity=1 ][line width=0.75]      (0, 0) circle [x radius= 3.35, y radius= 3.35]   ;
		\draw [shift={(321,144.6)}, rotate = 0] [color={rgb, 255:red, 0; green, 0; blue, 0 }  ,draw opacity=1 ][fill={rgb, 255:red, 0; green, 0; blue, 0 }  ,fill opacity=1 ][line width=0.75]      (0, 0) circle [x radius= 3.35, y radius= 3.35]   ;
		\draw [color={rgb, 255:red, 0; green, 0; blue, 0 }  ,draw opacity=1 ]   (361,144.6) -- (381,144.6) ;
		\draw [shift={(381,144.6)}, rotate = 0] [color={rgb, 255:red, 0; green, 0; blue, 0 }  ,draw opacity=1 ][fill={rgb, 255:red, 0; green, 0; blue, 0 }  ,fill opacity=1 ][line width=0.75]      (0, 0) circle [x radius= 3.35, y radius= 3.35]   ;
		\draw [shift={(361,144.6)}, rotate = 0] [color={rgb, 255:red, 0; green, 0; blue, 0 }  ,draw opacity=1 ][fill={rgb, 255:red, 0; green, 0; blue, 0 }  ,fill opacity=1 ][line width=0.75]      (0, 0) circle [x radius= 3.35, y radius= 3.35]   ;
		\draw [color={rgb, 255:red, 0; green, 0; blue, 0 }  ,draw opacity=1 ]   (341,101.6) -- (361,101.6) ;
		\draw [shift={(361,101.6)}, rotate = 0] [color={rgb, 255:red, 0; green, 0; blue, 0 }  ,draw opacity=1 ][fill={rgb, 255:red, 0; green, 0; blue, 0 }  ,fill opacity=1 ][line width=0.75]      (0, 0) circle [x radius= 3.35, y radius= 3.35]   ;
		\draw [shift={(341,101.6)}, rotate = 0] [color={rgb, 255:red, 0; green, 0; blue, 0 }  ,draw opacity=1 ][fill={rgb, 255:red, 0; green, 0; blue, 0 }  ,fill opacity=1 ][line width=0.75]      (0, 0) circle [x radius= 3.35, y radius= 3.35]   ;
		\draw [color={rgb, 255:red, 0; green, 0; blue, 0 }  ,draw opacity=1 ]   (342,167.6) -- (362,167.6) ;
		\draw [shift={(362,167.6)}, rotate = 0] [color={rgb, 255:red, 0; green, 0; blue, 0 }  ,draw opacity=1 ][fill={rgb, 255:red, 0; green, 0; blue, 0 }  ,fill opacity=1 ][line width=0.75]      (0, 0) circle [x radius= 3.35, y radius= 3.35]   ;
		\draw [shift={(342,167.6)}, rotate = 0] [color={rgb, 255:red, 0; green, 0; blue, 0 }  ,draw opacity=1 ][fill={rgb, 255:red, 0; green, 0; blue, 0 }  ,fill opacity=1 ][line width=0.75]      (0, 0) circle [x radius= 3.35, y radius= 3.35]   ;
		\draw    (301,144.6) -- (342,167.6) ;
		\draw    (321,144.6) -- (362,167.6) ;
		\draw    (341,144.6) -- (342,167.6) ;
		\draw    (361,144.6) -- (362,167.6) ;
		\draw    (381,144.6) -- (342,167.6) ;
		\draw    (401,144.6) -- (362,167.6) ;
		\draw    (341,101.6) -- (301,123.6) ;
		\draw    (361,101.6) -- (321,123.6) ;
		\draw    (341,101.6) -- (341,123.6) ;
		\draw    (361,101.6) -- (361,123.6) ;
		\draw    (341,101.6) -- (381,123.6) ;
		\draw    (361,101.6) -- (402,124.6) ;
		\draw  [dash pattern={on 0.84pt off 2.51pt}]  (413,124) -- (424,124.6) ;
		\draw  [dash pattern={on 0.84pt off 2.51pt}]  (413,146) -- (424,146.6) ;
		\draw  [dash pattern={on 0.84pt off 2.51pt}]  (390,106) -- (401,106.6) ;
		\draw  [dash pattern={on 0.84pt off 2.51pt}]  (389,165) -- (400,165.6) ;
		\draw  [color={rgb, 255:red, 208; green, 2; blue, 27 }  ,draw opacity=1 ] (334.5,101.6) .. controls (334.5,98.01) and (337.41,95.1) .. (341,95.1) .. controls (344.59,95.1) and (347.5,98.01) .. (347.5,101.6) .. controls (347.5,105.19) and (344.59,108.1) .. (341,108.1) .. controls (337.41,108.1) and (334.5,105.19) .. (334.5,101.6) -- cycle ;
		\draw  [color={rgb, 255:red, 208; green, 2; blue, 27 }  ,draw opacity=1 ] (354.5,101.6) .. controls (354.5,98.01) and (357.41,95.1) .. (361,95.1) .. controls (364.59,95.1) and (367.5,98.01) .. (367.5,101.6) .. controls (367.5,105.19) and (364.59,108.1) .. (361,108.1) .. controls (357.41,108.1) and (354.5,105.19) .. (354.5,101.6) -- cycle ;
		\draw  [color={rgb, 255:red, 208; green, 2; blue, 27 }  ,draw opacity=1 ] (335.5,167.6) .. controls (335.5,164.01) and (338.41,161.1) .. (342,161.1) .. controls (345.59,161.1) and (348.5,164.01) .. (348.5,167.6) .. controls (348.5,171.19) and (345.59,174.1) .. (342,174.1) .. controls (338.41,174.1) and (335.5,171.19) .. (335.5,167.6) -- cycle ;
		\draw  [color={rgb, 255:red, 208; green, 2; blue, 27 }  ,draw opacity=1 ] (355.5,167.6) .. controls (355.5,164.01) and (358.41,161.1) .. (362,161.1) .. controls (365.59,161.1) and (368.5,164.01) .. (368.5,167.6) .. controls (368.5,171.19) and (365.59,174.1) .. (362,174.1) .. controls (358.41,174.1) and (355.5,171.19) .. (355.5,167.6) -- cycle ;

	\end{tikzpicture}
	\caption{A compact countable graph. The vertices circled in red form a finite dominating set.}
	\label{exm3}
\end{figure}

\begin{remark}
	\text{}
\begin{enumerate}
	
	\item Compact infinite graphs are not locally finite, since they admit a finite dominating set.
	
	\item A preconvergence space $\langle X,L\rangle$ is locally compact if, for every convergent net $\varphi\in\textsc{Nets}(X)$, there exists a compact set $C\in \varphi^\uparrow$, i.e., $C$ is compact with respect to the subspace preconvergence. Notice that every graph is locally compact as a convergence space. In fact, the filter induced by a convergent net contains the neighborhood set of a vertex, which is a compact set under the induced subgraph convergence.
	
\end{enumerate}
\end{remark}

In ~\cite{Das2017} Angsuman Das characterizes infinite graphs that admit a finite dominating set. Then, we obtain a characterization of compact graphs.

\begin{theorem}[{\cite[Corollary 3.1]{Das2017}}]
	
	\label{domin}
	An infinite graph $G$ with finitely many components has a finite dominating set if and only if each component of $G$ has a spanning tree with finitely many internal vertices, i.e, vertices with degree greater than or equal to two.
\end{theorem}

\begin{remark}
In~\cite[Theorems 5.1 and 5.3]{Das2017}, Angsuman Das shows that if there exists a surjective graph homomorphism $f : G \to H$ and $G$ has a finite dominating set, then $H$ also has a finite dominating set. Moreover, he proves that the product of graphs each having a finite dominating set also admits a finite dominating set.

The proofs of these two results rely solely on combinatorial methods and can be simplified using the following:

\begin{itemize}
	\item Let $X$ and $Y$ be preconvergence spaces, where $X$ is compact. If there exists a surjective continuous function $f:X\to Y$, then $Y$ is compact. 
	\item The Tychonoff theorem holds for preconvergence spaces.
\end{itemize}
See \cite{BB} and \cite{dol} for proofs.

\end{remark}

When dealing with infinite graphs, we are interested not only in their vertices and edges, but also in their behavior ``at infinity.'' To capture this phenomenon, the notions of ends and edge-ends of an infinite graph are introduced. Our objective is to show that a compact graph has only finitely many edge-ends.

A ray $r$ in a graph $G$ is a sequence $\langle v_n\rangle_{n\in\mathbb{N}}$ of distinct vertices such that $v_nv_{n+1}\in E(G)$ for every $n\in\mathbb{N}$, and its infinite subpaths are called its \emph{tails}. Two rays $r$ and $s$ in a graph $G$ are \emph{equivalent} if for every finite set $F\subseteq V(G)$ there exists a path between tails of $r$ and $s$ in $G-F$ (see Figure~\ref{rays}), where $G-F$ denotes the subgraph with $V(G-F)=V(G)\setminus F$ and $E(G-F)=\{xy\in E(G): x\notin F \text{ and } y\notin F\}$. The corresponding equivalence classes of rays are the \emph{ends} of $G$, denoted by $\Omega(G)$. An end of an infinite graph is, intuitively, a direction in which the graph extends infinitely.

In 1931, Freudenthal~\cite{Freudenthal1931} published the article ``Uber die Enden topologischer R\"aume und Gruppen'', where he first introduced the concept of ends for topological spaces. Later, in 1964, Halin~\cite{Halin1964} introduced the concept of graph ends, motivated by problems in infinite networks and high connectivity. In the case of locally finite graphs, the ends of the graph coincide with the ends of the associated topological space obtained by viewing the graph as a 1-dimensional simplicial complex.
\begin{figure}[htb]
	\centering

	\tikzset{every picture/.style={line width=0.75pt}} 
	
	\begin{tikzpicture}[x=0.75pt,y=0.75pt,yscale=-1,xscale=1]
		
		\draw    (95,120) -- (122,120.6) ;
		\draw [shift={(122,120.6)}, rotate = 1.27] [color={rgb, 255:red, 0; green, 0; blue, 0 }  ][fill={rgb, 255:red, 0; green, 0; blue, 0 }  ][line width=0.75]      (0, 0) circle [x radius= 3.35, y radius= 3.35]   ;
		\draw [shift={(95,120)}, rotate = 1.27] [color={rgb, 255:red, 0; green, 0; blue, 0 }  ][fill={rgb, 255:red, 0; green, 0; blue, 0 }  ][line width=0.75]      (0, 0) circle [x radius= 3.35, y radius= 3.35]   ;
		\draw    (122,120.6) -- (149,121.2) ;
		\draw [shift={(149,121.2)}, rotate = 1.27] [color={rgb, 255:red, 0; green, 0; blue, 0 }  ][fill={rgb, 255:red, 0; green, 0; blue, 0 }  ][line width=0.75]      (0, 0) circle [x radius= 3.35, y radius= 3.35]   ;
		\draw [shift={(122,120.6)}, rotate = 1.27] [color={rgb, 255:red, 0; green, 0; blue, 0 }  ][fill={rgb, 255:red, 0; green, 0; blue, 0 }  ][line width=0.75]      (0, 0) circle [x radius= 3.35, y radius= 3.35]   ;
		\draw    (149,121.2) -- (176,121.8) ;
		\draw [shift={(176,121.8)}, rotate = 1.27] [color={rgb, 255:red, 0; green, 0; blue, 0 }  ][fill={rgb, 255:red, 0; green, 0; blue, 0 }  ][line width=0.75]      (0, 0) circle [x radius= 3.35, y radius= 3.35]   ;
		\draw [shift={(149,121.2)}, rotate = 1.27] [color={rgb, 255:red, 0; green, 0; blue, 0 }  ][fill={rgb, 255:red, 0; green, 0; blue, 0 }  ][line width=0.75]      (0, 0) circle [x radius= 3.35, y radius= 3.35]   ;
		\draw    (176,121.8) -- (203,122.4) ;
		\draw [shift={(203,122.4)}, rotate = 1.27] [color={rgb, 255:red, 0; green, 0; blue, 0 }  ][fill={rgb, 255:red, 0; green, 0; blue, 0 }  ][line width=0.75]      (0, 0) circle [x radius= 3.35, y radius= 3.35]   ;
		\draw [shift={(176,121.8)}, rotate = 1.27] [color={rgb, 255:red, 0; green, 0; blue, 0 }  ][fill={rgb, 255:red, 0; green, 0; blue, 0 }  ][line width=0.75]      (0, 0) circle [x radius= 3.35, y radius= 3.35]   ;
		\draw    (203,122.4) -- (230,123) ;
		\draw [shift={(230,123)}, rotate = 1.27] [color={rgb, 255:red, 0; green, 0; blue, 0 }  ][fill={rgb, 255:red, 0; green, 0; blue, 0 }  ][line width=0.75]      (0, 0) circle [x radius= 3.35, y radius= 3.35]   ;
		\draw [shift={(203,122.4)}, rotate = 1.27] [color={rgb, 255:red, 0; green, 0; blue, 0 }  ][fill={rgb, 255:red, 0; green, 0; blue, 0 }  ][line width=0.75]      (0, 0) circle [x radius= 3.35, y radius= 3.35]   ;
		\draw    (95,142) -- (122,142.6) ;
		\draw [shift={(122,142.6)}, rotate = 1.27] [color={rgb, 255:red, 0; green, 0; blue, 0 }  ][fill={rgb, 255:red, 0; green, 0; blue, 0 }  ][line width=0.75]      (0, 0) circle [x radius= 3.35, y radius= 3.35]   ;
		\draw [shift={(95,142)}, rotate = 1.27] [color={rgb, 255:red, 0; green, 0; blue, 0 }  ][fill={rgb, 255:red, 0; green, 0; blue, 0 }  ][line width=0.75]      (0, 0) circle [x radius= 3.35, y radius= 3.35]   ;
		\draw    (122,142.6) -- (149,143.2) ;
		\draw [shift={(149,143.2)}, rotate = 1.27] [color={rgb, 255:red, 0; green, 0; blue, 0 }  ][fill={rgb, 255:red, 0; green, 0; blue, 0 }  ][line width=0.75]      (0, 0) circle [x radius= 3.35, y radius= 3.35]   ;
		\draw [shift={(122,142.6)}, rotate = 1.27] [color={rgb, 255:red, 0; green, 0; blue, 0 }  ][fill={rgb, 255:red, 0; green, 0; blue, 0 }  ][line width=0.75]      (0, 0) circle [x radius= 3.35, y radius= 3.35]   ;
		\draw    (149,143.2) -- (176,143.8) ;
		\draw [shift={(176,143.8)}, rotate = 1.27] [color={rgb, 255:red, 0; green, 0; blue, 0 }  ][fill={rgb, 255:red, 0; green, 0; blue, 0 }  ][line width=0.75]      (0, 0) circle [x radius= 3.35, y radius= 3.35]   ;
		\draw [shift={(149,143.2)}, rotate = 1.27] [color={rgb, 255:red, 0; green, 0; blue, 0 }  ][fill={rgb, 255:red, 0; green, 0; blue, 0 }  ][line width=0.75]      (0, 0) circle [x radius= 3.35, y radius= 3.35]   ;
		\draw    (176,143.8) -- (203,144.4) ;
		\draw [shift={(203,144.4)}, rotate = 1.27] [color={rgb, 255:red, 0; green, 0; blue, 0 }  ][fill={rgb, 255:red, 0; green, 0; blue, 0 }  ][line width=0.75]      (0, 0) circle [x radius= 3.35, y radius= 3.35]   ;
		\draw [shift={(176,143.8)}, rotate = 1.27] [color={rgb, 255:red, 0; green, 0; blue, 0 }  ][fill={rgb, 255:red, 0; green, 0; blue, 0 }  ][line width=0.75]      (0, 0) circle [x radius= 3.35, y radius= 3.35]   ;
		\draw    (203,144.4) -- (230,145) ;
		\draw [shift={(230,145)}, rotate = 1.27] [color={rgb, 255:red, 0; green, 0; blue, 0 }  ][fill={rgb, 255:red, 0; green, 0; blue, 0 }  ][line width=0.75]      (0, 0) circle [x radius= 3.35, y radius= 3.35]   ;
		\draw [shift={(203,144.4)}, rotate = 1.27] [color={rgb, 255:red, 0; green, 0; blue, 0 }  ][fill={rgb, 255:red, 0; green, 0; blue, 0 }  ][line width=0.75]      (0, 0) circle [x radius= 3.35, y radius= 3.35]   ;
		\draw    (95,120) -- (95,142) ;
		\draw    (122,120.6) -- (122,142.6) ;
		\draw    (149,121.2) -- (149,143.2) ;
		\draw    (176,121.8) -- (176,143.8) ;
		\draw    (203,122.4) -- (203,144.4) ;
		\draw  [dash pattern={on 0.84pt off 2.51pt}]  (240,123) -- (256,122.6) ;
		\draw  [dash pattern={on 0.84pt off 2.51pt}]  (240,146) -- (256,145.6) ;
		\draw    (281,124) -- (308,124.6) ;
		\draw [shift={(308,124.6)}, rotate = 1.27] [color={rgb, 255:red, 0; green, 0; blue, 0 }  ][fill={rgb, 255:red, 0; green, 0; blue, 0 }  ][line width=0.75]      (0, 0) circle [x radius= 3.35, y radius= 3.35]   ;
		\draw [shift={(281,124)}, rotate = 1.27] [color={rgb, 255:red, 0; green, 0; blue, 0 }  ][fill={rgb, 255:red, 0; green, 0; blue, 0 }  ][line width=0.75]      (0, 0) circle [x radius= 3.35, y radius= 3.35]   ;
		\draw    (308,124.6) -- (335,125.2) ;
		\draw [shift={(335,125.2)}, rotate = 1.27] [color={rgb, 255:red, 0; green, 0; blue, 0 }  ][fill={rgb, 255:red, 0; green, 0; blue, 0 }  ][line width=0.75]      (0, 0) circle [x radius= 3.35, y radius= 3.35]   ;
		\draw [shift={(308,124.6)}, rotate = 1.27] [color={rgb, 255:red, 0; green, 0; blue, 0 }  ][fill={rgb, 255:red, 0; green, 0; blue, 0 }  ][line width=0.75]      (0, 0) circle [x radius= 3.35, y radius= 3.35]   ;
		\draw    (335,125.2) -- (362,125.8) ;
		\draw [shift={(362,125.8)}, rotate = 1.27] [color={rgb, 255:red, 0; green, 0; blue, 0 }  ][fill={rgb, 255:red, 0; green, 0; blue, 0 }  ][line width=0.75]      (0, 0) circle [x radius= 3.35, y radius= 3.35]   ;
		\draw [shift={(335,125.2)}, rotate = 1.27] [color={rgb, 255:red, 0; green, 0; blue, 0 }  ][fill={rgb, 255:red, 0; green, 0; blue, 0 }  ][line width=0.75]      (0, 0) circle [x radius= 3.35, y radius= 3.35]   ;
		\draw    (362,125.8) -- (389,126.4) ;
		\draw [shift={(389,126.4)}, rotate = 1.27] [color={rgb, 255:red, 0; green, 0; blue, 0 }  ][fill={rgb, 255:red, 0; green, 0; blue, 0 }  ][line width=0.75]      (0, 0) circle [x radius= 3.35, y radius= 3.35]   ;
		\draw [shift={(362,125.8)}, rotate = 1.27] [color={rgb, 255:red, 0; green, 0; blue, 0 }  ][fill={rgb, 255:red, 0; green, 0; blue, 0 }  ][line width=0.75]      (0, 0) circle [x radius= 3.35, y radius= 3.35]   ;
		\draw    (389,126.4) -- (416,127) ;
		\draw [shift={(416,127)}, rotate = 1.27] [color={rgb, 255:red, 0; green, 0; blue, 0 }  ][fill={rgb, 255:red, 0; green, 0; blue, 0 }  ][line width=0.75]      (0, 0) circle [x radius= 3.35, y radius= 3.35]   ;
		\draw [shift={(389,126.4)}, rotate = 1.27] [color={rgb, 255:red, 0; green, 0; blue, 0 }  ][fill={rgb, 255:red, 0; green, 0; blue, 0 }  ][line width=0.75]      (0, 0) circle [x radius= 3.35, y radius= 3.35]   ;
		\draw    (281,159) -- (308,159.6) ;
		\draw [shift={(308,159.6)}, rotate = 1.27] [color={rgb, 255:red, 0; green, 0; blue, 0 }  ][fill={rgb, 255:red, 0; green, 0; blue, 0 }  ][line width=0.75]      (0, 0) circle [x radius= 3.35, y radius= 3.35]   ;
		\draw [shift={(281,159)}, rotate = 1.27] [color={rgb, 255:red, 0; green, 0; blue, 0 }  ][fill={rgb, 255:red, 0; green, 0; blue, 0 }  ][line width=0.75]      (0, 0) circle [x radius= 3.35, y radius= 3.35]   ;
		\draw    (308,159.6) -- (335,160.2) ;
		\draw [shift={(335,160.2)}, rotate = 1.27] [color={rgb, 255:red, 0; green, 0; blue, 0 }  ][fill={rgb, 255:red, 0; green, 0; blue, 0 }  ][line width=0.75]      (0, 0) circle [x radius= 3.35, y radius= 3.35]   ;
		\draw [shift={(308,159.6)}, rotate = 1.27] [color={rgb, 255:red, 0; green, 0; blue, 0 }  ][fill={rgb, 255:red, 0; green, 0; blue, 0 }  ][line width=0.75]      (0, 0) circle [x radius= 3.35, y radius= 3.35]   ;
		\draw    (335,160.2) -- (362,160.8) ;
		\draw [shift={(362,160.8)}, rotate = 1.27] [color={rgb, 255:red, 0; green, 0; blue, 0 }  ][fill={rgb, 255:red, 0; green, 0; blue, 0 }  ][line width=0.75]      (0, 0) circle [x radius= 3.35, y radius= 3.35]   ;
		\draw [shift={(335,160.2)}, rotate = 1.27] [color={rgb, 255:red, 0; green, 0; blue, 0 }  ][fill={rgb, 255:red, 0; green, 0; blue, 0 }  ][line width=0.75]      (0, 0) circle [x radius= 3.35, y radius= 3.35]   ;
		\draw    (362,160.8) -- (389,161.4) ;
		\draw [shift={(389,161.4)}, rotate = 1.27] [color={rgb, 255:red, 0; green, 0; blue, 0 }  ][fill={rgb, 255:red, 0; green, 0; blue, 0 }  ][line width=0.75]      (0, 0) circle [x radius= 3.35, y radius= 3.35]   ;
		\draw [shift={(362,160.8)}, rotate = 1.27] [color={rgb, 255:red, 0; green, 0; blue, 0 }  ][fill={rgb, 255:red, 0; green, 0; blue, 0 }  ][line width=0.75]      (0, 0) circle [x radius= 3.35, y radius= 3.35]   ;
		\draw    (389,161.4) -- (416,162) ;
		\draw [shift={(416,162)}, rotate = 1.27] [color={rgb, 255:red, 0; green, 0; blue, 0 }  ][fill={rgb, 255:red, 0; green, 0; blue, 0 }  ][line width=0.75]      (0, 0) circle [x radius= 3.35, y radius= 3.35]   ;
		\draw [shift={(389,161.4)}, rotate = 1.27] [color={rgb, 255:red, 0; green, 0; blue, 0 }  ][fill={rgb, 255:red, 0; green, 0; blue, 0 }  ][line width=0.75]      (0, 0) circle [x radius= 3.35, y radius= 3.35]   ;
		\draw  [dash pattern={on 0.84pt off 2.51pt}]  (426,127) -- (442,126.6) ;
		\draw  [dash pattern={on 0.84pt off 2.51pt}]  (426,163) -- (442,162.6) ;
		\draw    (281,124) -- (349,144.6) ;
		\draw [shift={(349,144.6)}, rotate = 16.85] [color={rgb, 255:red, 0; green, 0; blue, 0 }  ][fill={rgb, 255:red, 0; green, 0; blue, 0 }  ][line width=0.75]      (0, 0) circle [x radius= 3.35, y radius= 3.35]   ;
		\draw [shift={(281,124)}, rotate = 16.85] [color={rgb, 255:red, 0; green, 0; blue, 0 }  ][fill={rgb, 255:red, 0; green, 0; blue, 0 }  ][line width=0.75]      (0, 0) circle [x radius= 3.35, y radius= 3.35]   ;
		\draw    (308,124.6) -- (349,144.6) ;
		\draw [shift={(349,144.6)}, rotate = 26] [color={rgb, 255:red, 0; green, 0; blue, 0 }  ][fill={rgb, 255:red, 0; green, 0; blue, 0 }  ][line width=0.75]      (0, 0) circle [x radius= 3.35, y radius= 3.35]   ;
		\draw [shift={(308,124.6)}, rotate = 26] [color={rgb, 255:red, 0; green, 0; blue, 0 }  ][fill={rgb, 255:red, 0; green, 0; blue, 0 }  ][line width=0.75]      (0, 0) circle [x radius= 3.35, y radius= 3.35]   ;
		\draw    (335,125.2) -- (349,144.6) ;
		\draw [shift={(349,144.6)}, rotate = 54.18] [color={rgb, 255:red, 0; green, 0; blue, 0 }  ][fill={rgb, 255:red, 0; green, 0; blue, 0 }  ][line width=0.75]      (0, 0) circle [x radius= 3.35, y radius= 3.35]   ;
		\draw [shift={(335,125.2)}, rotate = 54.18] [color={rgb, 255:red, 0; green, 0; blue, 0 }  ][fill={rgb, 255:red, 0; green, 0; blue, 0 }  ][line width=0.75]      (0, 0) circle [x radius= 3.35, y radius= 3.35]   ;
		\draw    (362,125.8) -- (349,144.6) ;
		\draw [shift={(349,144.6)}, rotate = 124.66] [color={rgb, 255:red, 0; green, 0; blue, 0 }  ][fill={rgb, 255:red, 0; green, 0; blue, 0 }  ][line width=0.75]      (0, 0) circle [x radius= 3.35, y radius= 3.35]   ;
		\draw [shift={(362,125.8)}, rotate = 124.66] [color={rgb, 255:red, 0; green, 0; blue, 0 }  ][fill={rgb, 255:red, 0; green, 0; blue, 0 }  ][line width=0.75]      (0, 0) circle [x radius= 3.35, y radius= 3.35]   ;
		\draw    (389,126.4) -- (349,144.6) ;
		\draw [shift={(349,144.6)}, rotate = 155.53] [color={rgb, 255:red, 0; green, 0; blue, 0 }  ][fill={rgb, 255:red, 0; green, 0; blue, 0 }  ][line width=0.75]      (0, 0) circle [x radius= 3.35, y radius= 3.35]   ;
		\draw [shift={(389,126.4)}, rotate = 155.53] [color={rgb, 255:red, 0; green, 0; blue, 0 }  ][fill={rgb, 255:red, 0; green, 0; blue, 0 }  ][line width=0.75]      (0, 0) circle [x radius= 3.35, y radius= 3.35]   ;
		\draw    (416,127) -- (349,144.6) ;
		\draw [shift={(349,144.6)}, rotate = 165.28] [color={rgb, 255:red, 0; green, 0; blue, 0 }  ][fill={rgb, 255:red, 0; green, 0; blue, 0 }  ][line width=0.75]      (0, 0) circle [x radius= 3.35, y radius= 3.35]   ;
		\draw [shift={(416,127)}, rotate = 165.28] [color={rgb, 255:red, 0; green, 0; blue, 0 }  ][fill={rgb, 255:red, 0; green, 0; blue, 0 }  ][line width=0.75]      (0, 0) circle [x radius= 3.35, y radius= 3.35]   ;
		\draw    (281,159) -- (349,144.6) ;
		\draw [shift={(349,144.6)}, rotate = 348.04] [color={rgb, 255:red, 0; green, 0; blue, 0 }  ][fill={rgb, 255:red, 0; green, 0; blue, 0 }  ][line width=0.75]      (0, 0) circle [x radius= 3.35, y radius= 3.35]   ;
		\draw [shift={(281,159)}, rotate = 348.04] [color={rgb, 255:red, 0; green, 0; blue, 0 }  ][fill={rgb, 255:red, 0; green, 0; blue, 0 }  ][line width=0.75]      (0, 0) circle [x radius= 3.35, y radius= 3.35]   ;
		\draw    (308,159.6) -- (349,144.6) ;
		\draw [shift={(349,144.6)}, rotate = 339.9] [color={rgb, 255:red, 0; green, 0; blue, 0 }  ][fill={rgb, 255:red, 0; green, 0; blue, 0 }  ][line width=0.75]      (0, 0) circle [x radius= 3.35, y radius= 3.35]   ;
		\draw [shift={(308,159.6)}, rotate = 339.9] [color={rgb, 255:red, 0; green, 0; blue, 0 }  ][fill={rgb, 255:red, 0; green, 0; blue, 0 }  ][line width=0.75]      (0, 0) circle [x radius= 3.35, y radius= 3.35]   ;
		\draw    (335,160.2) -- (349,144.6) ;
		\draw [shift={(349,144.6)}, rotate = 311.91] [color={rgb, 255:red, 0; green, 0; blue, 0 }  ][fill={rgb, 255:red, 0; green, 0; blue, 0 }  ][line width=0.75]      (0, 0) circle [x radius= 3.35, y radius= 3.35]   ;
		\draw [shift={(335,160.2)}, rotate = 311.91] [color={rgb, 255:red, 0; green, 0; blue, 0 }  ][fill={rgb, 255:red, 0; green, 0; blue, 0 }  ][line width=0.75]      (0, 0) circle [x radius= 3.35, y radius= 3.35]   ;
		\draw    (362,160.8) -- (349,144.6) ;
		\draw [shift={(349,144.6)}, rotate = 231.25] [color={rgb, 255:red, 0; green, 0; blue, 0 }  ][fill={rgb, 255:red, 0; green, 0; blue, 0 }  ][line width=0.75]      (0, 0) circle [x radius= 3.35, y radius= 3.35]   ;
		\draw [shift={(362,160.8)}, rotate = 231.25] [color={rgb, 255:red, 0; green, 0; blue, 0 }  ][fill={rgb, 255:red, 0; green, 0; blue, 0 }  ][line width=0.75]      (0, 0) circle [x radius= 3.35, y radius= 3.35]   ;
		\draw    (349,144.6) -- (389,161.4) ;
		\draw [shift={(389,161.4)}, rotate = 22.78] [color={rgb, 255:red, 0; green, 0; blue, 0 }  ][fill={rgb, 255:red, 0; green, 0; blue, 0 }  ][line width=0.75]      (0, 0) circle [x radius= 3.35, y radius= 3.35]   ;
		\draw [shift={(349,144.6)}, rotate = 22.78] [color={rgb, 255:red, 0; green, 0; blue, 0 }  ][fill={rgb, 255:red, 0; green, 0; blue, 0 }  ][line width=0.75]      (0, 0) circle [x radius= 3.35, y radius= 3.35]   ;
		\draw    (349,144.6) -- (416,162) ;
		\draw [shift={(416,162)}, rotate = 14.56] [color={rgb, 255:red, 0; green, 0; blue, 0 }  ][fill={rgb, 255:red, 0; green, 0; blue, 0 }  ][line width=0.75]      (0, 0) circle [x radius= 3.35, y radius= 3.35]   ;
		\draw [shift={(349,144.6)}, rotate = 14.56] [color={rgb, 255:red, 0; green, 0; blue, 0 }  ][fill={rgb, 255:red, 0; green, 0; blue, 0 }  ][line width=0.75]      (0, 0) circle [x radius= 3.35, y radius= 3.35]   ;
		\draw  [dash pattern={on 0.84pt off 2.51pt}]  (240,135) -- (256,134.6) ;
		\draw  [dash pattern={on 0.84pt off 2.51pt}]  (426,145) -- (442,144.6) ;

	\end{tikzpicture}
	\label{rays}
	\caption{The rays on the left are equivalent, whereas those on the right are not.}
\end{figure}

When edge-connectivity is taken as the central notion, one is naturally led to the concept of edge-end spaces. 
Two rays $r$ and $s$ in a graph $G$ are said to be \emph{edge-equivalent} if for every finite set $F\subseteq E(G)$ there exists a path connecting tails of $r$ and $s$ in $G-F$, where $G-F$ denotes the subgraph with $V(G-F)=V(G)$ and $E(G-F)=E(G)\setminus F$. The corresponding equivalence classes of rays are called the \emph{edge-ends} of $G$, denoted by $\Omega_E(G)$.

Note that the rays on the right in Figure~\ref{rays} are edge-equivalent, although they are not equivalent under the vertex-based definition of ends. Edge-end spaces were originally introduced in~\cite{Hahn1997} and have recently been studied in~\cite{aurichi2024topologicalremarksendedgeend} and~\cite{boska2025edgedirectioncompactedgeendspaces}.

\begin{corollary}
	\label{edge}
	The edge-end space $\Omega_E(G)$ of a compact graph $G$ is finite.
\end{corollary}

\begin{proof}
	By Theorem~\ref{teocomp2}, there exists a finite dominating set
	$D\subseteq V(G)$. Suppose for a contradiction that $G$ has infinitely
	many edge-ends. Choose more than $|D|$ distinct edge-ends and take a
	ray from each of them.
	
	Since $D$ is dominating, every vertex of a ray has a neighbour in $D$.
	As each ray contains infinitely many vertices and $D$ is finite, there
	exists a vertex $v\in D$ adjacent to infinitely many vertices of the ray.
	
	Because the number of chosen rays exceeds $|D|$, two of them,
	say $r$ and $s$, have infinitely many vertices adjacent to the
	same vertex $v\in D$.
	
	Let $F$ be a finite set of edges. Since $v$ is adjacent to infinitely
	many vertices of each ray, we can choose vertices $x\in r$ and
	$y\in s$ such that the edges $vx$ and $vy$ do not belong to $F$.
	Then the path $xvy$ lies in $G-F$. In particular, the tails of
	$r$ and $s$ in $G-F$ lie in the same component.
	
	This contradicts the choice of the rays from distinct edge-ends.
	Therefore $\Omega_E(G)$ is finite.
\end{proof}

The Figure \ref{exm} shows that Corollary \ref{edge} fails when considering the end space.

\begin{figure}[h]
	\centering

	\tikzset{every picture/.style={line width=0.75pt}} 
	
	\begin{tikzpicture}[x=0.75pt,y=0.75pt,yscale=-1,xscale=1]
		
		\draw    (100,130) -- (85,145.6) ;
		\draw [shift={(85,145.6)}, rotate = 133.88] [color={rgb, 255:red, 0; green, 0; blue, 0 }  ][fill={rgb, 255:red, 0; green, 0; blue, 0 }  ][line width=0.75]      (0, 0) circle [x radius= 3.35, y radius= 3.35]   ;
		\draw [shift={(100,130)}, rotate = 133.88] [color={rgb, 255:red, 0; green, 0; blue, 0 }  ][fill={rgb, 255:red, 0; green, 0; blue, 0 }  ][line width=0.75]      (0, 0) circle [x radius= 3.35, y radius= 3.35]   ;
		\draw    (115,114.4) -- (100,130) ;
		\draw [shift={(100,130)}, rotate = 133.88] [color={rgb, 255:red, 0; green, 0; blue, 0 }  ][fill={rgb, 255:red, 0; green, 0; blue, 0 }  ][line width=0.75]      (0, 0) circle [x radius= 3.35, y radius= 3.35]   ;
		\draw [shift={(115,114.4)}, rotate = 133.88] [color={rgb, 255:red, 0; green, 0; blue, 0 }  ][fill={rgb, 255:red, 0; green, 0; blue, 0 }  ][line width=0.75]      (0, 0) c
		ircle [x radius= 3.35, y radius= 3.35]   ;
		\draw    (120,150) -- (105,165.6) ;
		\draw [shift={(105,165.6)}, rotate = 133.88] [color={rgb, 255:red, 0; green, 0; blue, 0 }  ][fill={rgb, 255:red, 0; green, 0; blue, 0 }  ][line width=0.75]      (0, 0) circle [x radius= 3.35, y radius= 3.35]   ;
		\draw [shift={(120,150)}, rotate = 133.88] [color={rgb, 255:red, 0; green, 0; blue, 0 }  ][fill={rgb, 255:red, 0; green, 0; blue, 0 }  ][line width=0.75]      (0, 0) circle [x radius= 3.35, y radius= 3.35]   ;
		\draw    (135,134.4) -- (120,150) ;
		\draw [shift={(120,150)}, rotate = 133.88] [color={rgb, 255:red, 0; green, 0; blue, 0 }  ][fill={rgb, 255:red, 0; green, 0; blue, 0 }  ][line width=0.75]      (0, 0) circle [x radius= 3.35, y radius= 3.35]   ;
		\draw [shift={(135,134.4)}, rotate = 133.88] [color={rgb, 255:red, 0; green, 0; blue, 0 }  ][fill={rgb, 255:red, 0; green, 0; blue, 0 }  ][line width=0.75]      (0, 0) circle [x radius= 3.35, y radius= 3.35]   ;
		\draw    (145,168) -- (130,183.6) ;
		\draw [shift={(130,183.6)}, rotate = 133.88] [color={rgb, 255:red, 0; green, 0; blue, 0 }  ][fill={rgb, 255:red, 0; green, 0; blue, 0 }  ][line width=0.75]      (0, 0) circle [x radius= 3.35, y radius= 3.35]   ;
		\draw [shift={(145,168)}, rotate = 133.88] [color={rgb, 255:red, 0; green, 0; blue, 0 }  ][fill={rgb, 255:red, 0; green, 0; blue, 0 }  ][line width=0.75]      (0, 0) circle [x radius= 3.35, y radius= 3.35]   ;
		\draw    (160,152.4) -- (145,168) ;
		\draw [shift={(145,168)}, rotate = 133.88] [color={rgb, 255:red, 0; green, 0; blue, 0 }  ][fill={rgb, 255:red, 0; green, 0; blue, 0 }  ][line width=0.75]      (0, 0) circle [x radius= 3.35, y radius= 3.35]   ;
		\draw [shift={(160,152.4)}, rotate = 133.88] [color={rgb, 255:red, 0; green, 0; blue, 0 }  ][fill={rgb, 255:red, 0; green, 0; blue, 0 }  ][line width=0.75]      (0, 0) circle [x radius= 3.35, y radius= 3.35]   ;
		\draw    (165,188) -- (150,203.6) ;
		\draw [shift={(150,203.6)}, rotate = 133.88] [color={rgb, 255:red, 0; green, 0; blue, 0 }  ][fill={rgb, 255:red, 0; green, 0; blue, 0 }  ][line width=0.75]      (0, 0) circle [x radius= 3.35, y radius= 3.35]   ;
		\draw [shift={(165,188)}, rotate = 133.88] [color={rgb, 255:red, 0; green, 0; blue, 0 }  ][fill={rgb, 255:red, 0; green, 0; blue, 0 }  ][line width=0.75]      (0, 0) circle [x radius= 3.35, y radius= 3.35]   ;
		\draw    (180,172.4) -- (165,188) ;
		\draw [shift={(165,188)}, rotate = 133.88] [color={rgb, 255:red, 0; green, 0; blue, 0 }  ][fill={rgb, 255:red, 0; green, 0; blue, 0 }  ][line width=0.75]      (0, 0) circle [x radius= 3.35, y radius= 3.35]   ;
		\draw [shift={(180,172.4)}, rotate = 133.88] [color={rgb, 255:red, 0; green, 0; blue, 0 }  ][fill={rgb, 255:red, 0; green, 0; blue, 0 }  ][line width=0.75]      (0, 0) circle [x radius= 3.35, y radius= 3.35]   ;
		\draw    (85,145.6) .. controls (68,171.6) and (71,189.6) .. (95,201.6) ;
		\draw [shift={(95,201.6)}, rotate = 26.57] [color={rgb, 255:red, 0; green, 0; blue, 0 }  ][fill={rgb, 255:red, 0; green, 0; blue, 0 }  ][line width=0.75]      (0, 0) circle [x radius= 3.35, y radius= 3.35]   ;
		\draw [shift={(85,145.6)}, rotate = 123.18] [color={rgb, 255:red, 0; green, 0; blue, 0 }  ][fill={rgb, 255:red, 0; green, 0; blue, 0 }  ][line width=0.75]      (0, 0) circle [x radius= 3.35, y radius= 3.35]   ;
		\draw    (105,165.6) .. controls (118,171.6) and (107,188.6) .. (95,201.6) ;
		\draw [shift={(95,201.6)}, rotate = 132.71] [color={rgb, 255:red, 0; green, 0; blue, 0 }  ][fill={rgb, 255:red, 0; green, 0; blue, 0 }  ][line width=0.75]      (0, 0) circle [x radius= 3.35, y radius= 3.35]   ;
		\draw [shift={(105,165.6)}, rotate = 24.78] [color={rgb, 255:red, 0; green, 0; blue, 0 }  ][fill={rgb, 255:red, 0; green, 0; blue, 0 }  ][line width=0.75]      (0, 0) circle [x radius= 3.35, y radius= 3.35]   ;
		\draw    (130,183.6) .. controls (131,187.6) and (124,201.6) .. (95,201.6) ;
		\draw [shift={(95,201.6)}, rotate = 180] [color={rgb, 255:red, 0; green, 0; blue, 0 }  ][fill={rgb, 255:red, 0; green, 0; blue, 0 }  ][line width=0.75]      (0, 0) circle [x radius= 3.35, y radius= 3.35]   ;
		\draw [shift={(130,183.6)}, rotate = 75.96] [color={rgb, 255:red, 0; green, 0; blue, 0 }  ][fill={rgb, 255:red, 0; green, 0; blue, 0 }  ][line width=0.75]      (0, 0) circle [x radius= 3.35, y radius= 3.35]   ;
		\draw    (150,203.6) .. controls (151,206.6) and (137,232.6) .. (95,201.6) ;
		\draw [shift={(95,201.6)}, rotate = 216.43] [color={rgb, 255:red, 0; green, 0; blue, 0 }  ][fill={rgb, 255:red, 0; green, 0; blue, 0 }  ][line width=0.75]      (0, 0) circle [x radius= 3.35, y radius= 3.35]   ;
		\draw [shift={(150,203.6)}, rotate = 71.57] [color={rgb, 255:red, 0; green, 0; blue, 0 }  ][fill={rgb, 255:red, 0; green, 0; blue, 0 }  ][line width=0.75]      (0, 0) circle [x radius= 3.35, y radius= 3.35]   ;
		\draw    (100,130) .. controls (87,156.6) and (77,170.6) .. (95,201.6) ;
		\draw [shift={(95,201.6)}, rotate = 59.86] [color={rgb, 255:red, 0; green, 0; blue, 0 }  ][fill={rgb, 255:red, 0; green, 0; blue, 0 }  ][line width=0.75]      (0, 0) circle [x radius= 3.35, y radius= 3.35]   ;
		\draw [shift={(100,130)}, rotate = 116.05] [color={rgb, 255:red, 0; green, 0; blue, 0 }  ][fill={rgb, 255:red, 0; green, 0; blue, 0 }  ][line width=0.75]      (0, 0) circle [x radius= 3.35, y radius= 3.35]   ;
		\draw    (120,150) .. controls (132,175.6) and (110,190.6) .. (95,201.6) ;
		\draw [shift={(95,201.6)}, rotate = 143.75] [color={rgb, 255:red, 0; green, 0; blue, 0 }  ][fill={rgb, 255:red, 0; green, 0; blue, 0 }  ][line width=0.75]      (0, 0) circle [x radius= 3.35, y radius= 3.35]   ;
		\draw [shift={(120,150)}, rotate = 64.89] [color={rgb, 255:red, 0; green, 0; blue, 0 }  ][fill={rgb, 255:red, 0; green, 0; blue, 0 }  ][line width=0.75]      (0, 0) circle [x radius= 3.35, y radius= 3.35]   ;
		\draw    (145,168) .. controls (142,165.6) and (153,212.6) .. (95,201.6) ;
		\draw [shift={(95,201.6)}, rotate = 190.74] [color={rgb, 255:red, 0; green, 0; blue, 0 }  ][fill={rgb, 255:red, 0; green, 0; blue, 0 }  ][line width=0.75]      (0, 0) circle [x radius= 3.35, y radius= 3.35]   ;
		\draw [shift={(145,168)}, rotate = 218.66] [color={rgb, 255:red, 0; green, 0; blue, 0 }  ][fill={rgb, 255:red, 0; green, 0; blue, 0 }  ][line width=0.75]      (0, 0) circle [x radius= 3.35, y radius= 3.35]   ;
		\draw    (165,188) .. controls (172,190.6) and (182,235.6) .. (95,201.6) ;
		\draw [shift={(95,201.6)}, rotate = 201.35] [color={rgb, 255:red, 0; green, 0; blue, 0 }  ][fill={rgb, 255:red, 0; green, 0; blue, 0 }  ][line width=0.75]      (0, 0) circle [x radius= 3.35, y radius= 3.35]   ;
		\draw [shift={(165,188)}, rotate = 20.38] [color={rgb, 255:red, 0; green, 0; blue, 0 }  ][fill={rgb, 255:red, 0; green, 0; blue, 0 }  ][line width=0.75]      (0, 0) circle [x radius= 3.35, y radius= 3.35]   ;
		\draw    (115,114.4) .. controls (121,117.6) and (88,128.6) .. (95,201.6) ;
		\draw [shift={(95,201.6)}, rotate = 84.52] [color={rgb, 255:red, 0; green, 0; blue, 0 }  ][fill={rgb, 255:red, 0; green, 0; blue, 0 }  ][line width=0.75]      (0, 0) circle [x radius= 3.35, y radius= 3.35]   ;
		\draw [shift={(115,114.4)}, rotate = 28.07] [color={rgb, 255:red, 0; green, 0; blue, 0 }  ][fill={rgb, 255:red, 0; green, 0; blue, 0 }  ][line width=0.75]      (0, 0) circle [x radius= 3.35, y radius= 3.35]   ;
		\draw    (135,134.4) .. controls (137,167.6) and (137,173.6) .. (95,201.6) ;
		\draw [shift={(95,201.6)}, rotate = 146.31] [color={rgb, 255:red, 0; green, 0; blue, 0 }  ][fill={rgb, 255:red, 0; green, 0; blue, 0 }  ][line width=0.75]      (0, 0) circle [x radius= 3.35, y radius= 3.35]   ;
		\draw [shift={(135,134.4)}, rotate = 86.55] [color={rgb, 255:red, 0; green, 0; blue, 0 }  ][fill={rgb, 255:red, 0; green, 0; blue, 0 }  ][line width=0.75]      (0, 0) circle [x radius= 3.35, y radius= 3.35]   ;
		\draw    (160,152.4) .. controls (160,203.6) and (128,199.6) .. (95,201.6) ;
		\draw [shift={(95,201.6)}, rotate = 176.53] [color={rgb, 255:red, 0; green, 0; blue, 0 }  ][fill={rgb, 255:red, 0; green, 0; blue, 0 }  ][line width=0.75]      (0, 0) circle [x radius= 3.35, y radius= 3.35]   ;
		\draw [shift={(160,152.4)}, rotate = 90] [color={rgb, 255:red, 0; green, 0; blue, 0 }  ][fill={rgb, 255:red, 0; green, 0; blue, 0 }  ][line width=0.75]      (0, 0) circle [x radius= 3.35, y radius= 3.35]   ;
		\draw    (180,172.4) .. controls (187,188.6) and (198,239.6) .. (95,201.6) ;
		\draw [shift={(95,201.6)}, rotate = 200.25] [color={rgb, 255:red, 0; green, 0; blue, 0 }  ][fill={rgb, 255:red, 0; green, 0; blue, 0 }  ][line width=0.75]      (0, 0) circle [x radius= 3.35, y radius= 3.35]   ;
		\draw [shift={(180,172.4)}, rotate = 66.63] [color={rgb, 255:red, 0; green, 0; blue, 0 }  ][fill={rgb, 255:red, 0; green, 0; blue, 0 }  ][line width=0.75]      (0, 0) circle [x radius= 3.35, y radius= 3.35]   ;
		\draw  [dash pattern={on 0.84pt off 2.51pt}]  (121,109) -- (128,101.6) ;
		\draw  [dash pattern={on 0.84pt off 2.51pt}]  (141,129) -- (148,121.6) ;
		\draw  [dash pattern={on 0.84pt off 2.51pt}]  (165,147) -- (172,139.6) ;
		\draw  [dash pattern={on 0.84pt off 2.51pt}]  (187,164) -- (194,156.6) ;
		\draw  [dash pattern={on 0.84pt off 2.51pt}]  (186,208) -- (196,215.6) ;

	\end{tikzpicture}

	\caption{A compact graph with infinitely many ends. }
		\label{exm}
\end{figure}
\begin{remark}
	\text{}
	\begin{enumerate}
\item 	The proof of Corollary~\ref{edge} establishes more than the finiteness of the edge-end space. Indeed, we obtain the bound $|\Omega_E(G)| \leq |D|$, where $D$ is a finite dominating set.
	
\item 	The Figure~\ref{rays} shows that the converse of Corollary~\ref{edge} fails.
\end{enumerate}
\end{remark}

\begin{corollary}
	Let $G$ be a graph. If $G$ is compact, then $G$ has a rayless spanning tree.
	\label{rayless}

	\end{corollary}
	
	\begin{proof}
By Theorem~\ref{teocomp2}, the graph $G$ has a finite dominating set. Observe that $G$ has finitely many connected components. If $G$ is finite, there is nothing to prove. Suppose, then, that $G$ is infinite. Consider the spanning tree with finitely many internal vertices provided by Theorem~\ref{domin}. Since the set of internal vertices is finite, this spanning tree cannot contain a ray. Therefore, there exists a rayless spanning tree in $G$.
	\end{proof}
	
	The Figure \ref{exm2} shows that the converse of Corollary \ref{rayless} also does not hold.
	
	\begin{figure}[H]
		\centering

		\tikzset{every picture/.style={line width=0.75pt}} 
		
		\begin{tikzpicture}[x=0.75pt,y=0.75pt,yscale=-1,xscale=1]
			
			\draw    (100,124) -- (122,149.6) ;
			\draw [shift={(122,149.6)}, rotate = 49.33] [color={rgb, 255:red, 0; green, 0; blue, 0 }  ][fill={rgb, 255:red, 0; green, 0; blue, 0 }  ][line width=0.75]      (0, 0) circle [x radius= 3.35, y radius= 3.35]   ;
			\draw [shift={(100,124)}, rotate = 49.33] [color={rgb, 255:red, 0; green, 0; blue, 0 }  ][fill={rgb, 255:red, 0; green, 0; blue, 0 }  ][line width=0.75]      (0, 0) circle [x radius= 3.35, y radius= 3.35]   ;
			\draw    (122,149.6) -- (147,133.6) ;
			\draw [shift={(147,133.6)}, rotate = 327.38] [color={rgb, 255:red, 0; green, 0; blue, 0 }  ][fill={rgb, 255:red, 0; green, 0; blue, 0 }  ][line width=0.75]      (0, 0) circle [x radius= 3.35, y radius= 3.35]   ;
			\draw [shift={(122,149.6)}, rotate = 327.38] [color={rgb, 255:red, 0; green, 0; blue, 0 }  ][fill={rgb, 255:red, 0; green, 0; blue, 0 }  ][line width=0.75]      (0, 0) circle [x radius= 3.35, y radius= 3.35]   ;
			\draw    (124,118.6) -- (122,149.6) ;
			\draw [shift={(122,149.6)}, rotate = 93.69] [color={rgb, 255:red, 0; green, 0; blue, 0 }  ][fill={rgb, 255:red, 0; green, 0; blue, 0 }  ][line width=0.75]      (0, 0) circle [x radius= 3.35, y radius= 3.35]   ;
			\draw [shift={(124,118.6)}, rotate = 93.69] [color={rgb, 255:red, 0; green, 0; blue, 0 }  ][fill={rgb, 255:red, 0; green, 0; blue, 0 }  ][line width=0.75]      (0, 0) circle [x radius= 3.35, y radius= 3.35]   ;
			\draw    (122,149.6) -- (153,159.6) ;
			\draw [shift={(153,159.6)}, rotate = 17.88] [color={rgb, 255:red, 0; green, 0; blue, 0 }  ][fill={rgb, 255:red, 0; green, 0; blue, 0 }  ][line width=0.75]      (0, 0) circle [x radius= 3.35, y radius= 3.35]   ;
			\draw [shift={(122,149.6)}, rotate = 17.88] [color={rgb, 255:red, 0; green, 0; blue, 0 }  ][fill={rgb, 255:red, 0; green, 0; blue, 0 }  ][line width=0.75]      (0, 0) circle [x radius= 3.35, y radius= 3.35]   ;
			\draw    (122,149.6) -- (129,175.6) ;
			\draw [shift={(129,175.6)}, rotate = 74.93] [color={rgb, 255:red, 0; green, 0; blue, 0 }  ][fill={rgb, 255:red, 0; green, 0; blue, 0 }  ][line width=0.75]      (0, 0) circle [x radius= 3.35, y radius= 3.35]   ;
			\draw [shift={(122,149.6)}, rotate = 74.93] [color={rgb, 255:red, 0; green, 0; blue, 0 }  ][fill={rgb, 255:red, 0; green, 0; blue, 0 }  ][line width=0.75]      (0, 0) circle [x radius= 3.35, y radius= 3.35]   ;
			\draw  [dash pattern={on 0.84pt off 2.51pt}]  (113,173.6) .. controls (101,171.6) and (84,138.6) .. (91,141.6) ;
			\draw    (80,117.6) -- (100,124) ;
			\draw [shift={(100,124)}, rotate = 17.74] [color={rgb, 255:red, 0; green, 0; blue, 0 }  ][fill={rgb, 255:red, 0; green, 0; blue, 0 }  ][line width=0.75]      (0, 0) circle [x radius= 3.35, y radius= 3.35]   ;
			\draw [shift={(80,117.6)}, rotate = 17.74] [color={rgb, 255:red, 0; green, 0; blue, 0 }  ][fill={rgb, 255:red, 0; green, 0; blue, 0 }  ][line width=0.75]      (0, 0) circle [x radius= 3.35, y radius= 3.35]   ;
			\draw    (101,105.6) -- (100,124) ;
			\draw [shift={(100,124)}, rotate = 93.11] [color={rgb, 255:red, 0; green, 0; blue, 0 }  ][fill={rgb, 255:red, 0; green, 0; blue, 0 }  ][line width=0.75]      (0, 0) circle [x radius= 3.35, y radius= 3.35]   ;
			\draw [shift={(101,105.6)}, rotate = 93.11] [color={rgb, 255:red, 0; green, 0; blue, 0 }  ][fill={rgb, 255:red, 0; green, 0; blue, 0 }  ][line width=0.75]      (0, 0) circle [x radius= 3.35, y radius= 3.35]   ;
			\draw    (116,100.6) -- (124,118.6) ;
			\draw [shift={(124,118.6)}, rotate = 66.04] [color={rgb, 255:red, 0; green, 0; blue, 0 }  ][fill={rgb, 255:red, 0; green, 0; blue, 0 }  ][line width=0.75]      (0, 0) circle [x radius= 3.35, y radius= 3.35]   ;
			\draw [shift={(116,100.6)}, rotate = 66.04] [color={rgb, 255:red, 0; green, 0; blue, 0 }  ][fill={rgb, 255:red, 0; green, 0; blue, 0 }  ][line width=0.75]      (0, 0) circle [x radius= 3.35, y radius= 3.35]   ;
			\draw    (124,118.6) -- (138,105.6) ;
			\draw [shift={(138,105.6)}, rotate = 317.12] [color={rgb, 255:red, 0; green, 0; blue, 0 }  ][fill={rgb, 255:red, 0; green, 0; blue, 0 }  ][line width=0.75]      (0, 0) circle [x radius= 3.35, y radius= 3.35]   ;
			\draw [shift={(124,118.6)}, rotate = 317.12] [color={rgb, 255:red, 0; green, 0; blue, 0 }  ][fill={rgb, 255:red, 0; green, 0; blue, 0 }  ][line width=0.75]      (0, 0) circle [x radius= 3.35, y radius= 3.35]   ;
			\draw    (147,133.6) -- (149,116.6) ;
			\draw [shift={(149,116.6)}, rotate = 276.71] [color={rgb, 255:red, 0; green, 0; blue, 0 }  ][fill={rgb, 255:red, 0; green, 0; blue, 0 }  ][line width=0.75]      (0, 0) circle [x radius= 3.35, y radius= 3.35]   ;
			\draw [shift={(147,133.6)}, rotate = 276.71] [color={rgb, 255:red, 0; green, 0; blue, 0 }  ][fill={rgb, 255:red, 0; green, 0; blue, 0 }  ][line width=0.75]      (0, 0) circle [x radius= 3.35, y radius= 3.35]   ;
			\draw    (147,133.6) -- (164,131.6) ;
			\draw [shift={(164,131.6)}, rotate = 353.29] [color={rgb, 255:red, 0; green, 0; blue, 0 }  ][fill={rgb, 255:red, 0; green, 0; blue, 0 }  ][line width=0.75]      (0, 0) circle [x radius= 3.35, y radius= 3.35]   ;
			\draw [shift={(147,133.6)}, rotate = 353.29] [color={rgb, 255:red, 0; green, 0; blue, 0 }  ][fill={rgb, 255:red, 0; green, 0; blue, 0 }  ][line width=0.75]      (0, 0) circle [x radius= 3.35, y radius= 3.35]   ;
			\draw    (153,159.6) -- (167,147.6) ;
			\draw [shift={(167,147.6)}, rotate = 319.4] [color={rgb, 255:red, 0; green, 0; blue, 0 }  ][fill={rgb, 255:red, 0; green, 0; blue, 0 }  ][line width=0.75]      (0, 0) circle [x radius= 3.35, y radius= 3.35]   ;
			\draw [shift={(153,159.6)}, rotate = 319.4] [color={rgb, 255:red, 0; green, 0; blue, 0 }  ][fill={rgb, 255:red, 0; green, 0; blue, 0 }  ][line width=0.75]      (0, 0) circle [x radius= 3.35, y radius= 3.35]   ;
			\draw    (153,159.6) -- (167,169.6) ;
			\draw [shift={(167,169.6)}, rotate = 35.54] [color={rgb, 255:red, 0; green, 0; blue, 0 }  ][fill={rgb, 255:red, 0; green, 0; blue, 0 }  ][line width=0.75]      (0, 0) circle [x radius= 3.35, y radius= 3.35]   ;
			\draw [shift={(153,159.6)}, rotate = 35.54] [color={rgb, 255:red, 0; green, 0; blue, 0 }  ][fill={rgb, 255:red, 0; green, 0; blue, 0 }  ][line width=0.75]      (0, 0) circle [x radius= 3.35, y radius= 3.35]   ;
			\draw    (129,175.6) -- (144,188.6) ;
			\draw [shift={(144,188.6)}, rotate = 40.91] [color={rgb, 255:red, 0; green, 0; blue, 0 }  ][fill={rgb, 255:red, 0; green, 0; blue, 0 }  ][line width=0.75]      (0, 0) circle [x radius= 3.35, y radius= 3.35]   ;
			\draw [shift={(129,175.6)}, rotate = 40.91] [color={rgb, 255:red, 0; green, 0; blue, 0 }  ][fill={rgb, 255:red, 0; green, 0; blue, 0 }  ][line width=0.75]      (0, 0) circle [x radius= 3.35, y radius= 3.35]   ;
			\draw    (129,175.6) -- (123,190.6) ;
			\draw [shift={(123,190.6)}, rotate = 111.8] [color={rgb, 255:red, 0; green, 0; blue, 0 }  ][fill={rgb, 255:red, 0; green, 0; blue, 0 }  ][line width=0.75]      (0, 0) circle [x radius= 3.35, y radius= 3.35]   ;
			\draw [shift={(129,175.6)}, rotate = 111.8] [color={rgb, 255:red, 0; green, 0; blue, 0 }  ][fill={rgb, 255:red, 0; green, 0; blue, 0 }  ][line width=0.75]      (0, 0) circle [x radius= 3.35, y radius= 3.35]   ;

		\end{tikzpicture}
		\caption{A rayless tree which is not compact.}
		\label{exm2}
	\end{figure}
	
		 The Corollary~\ref{rayless}, for connected countable graphs, can also be obtained as a consequence of a result due to \cite{kurkofka}, which states that a countable connected graph admits a spanning tree without rays if and only if every ray is dominated. Here, a ray is said to be \emph{dominated} if there exists a vertex $v$ such that infinitely many pairwise vertex-disjoint paths join $v$ to vertices of the ray. In particular, this applies whenever the graph has a finite dominating set, since in that case every ray is dominated.

\section{Remarks and future works}

So far, we have only considered convergence on the vertex set of a graph. However, many interesting combinatorial properties concern edges, so it is also important to define convergence for edges. A natural approach to this is to use the line graph. The \textit{line graph} of $G$, denoted by $L(G)$, is the graph whose vertex set is $E(G)$, where two vertices are adjacent if and only if the corresponding edges share an endpoint in $G$.
The concept of thorn graphs was introduced by Gutman and has since been studied in various applications by several authors (see~\cite{BeinekeBagga2021, boska2025edgedirectioncompactedgeendspaces,Chudnovsky2006}). Could this type of convergence provide a way to characterize edge-connectivity or to approach edge decompositions via convergence systems?

A next step would be to consider a convergence on the end space of a graph $G$ that is not the usual one but still retains some known properties. Recall that if $X \subseteq V(G)$ is finite and $\omega \in \Omega(G)$, there is a unique component of $G - X$ that contains a tail of every ray in $\omega$; we denote this component by $C(X,\omega)$. If $C$ is any component of $G - X$, we write $\Omega(X,C)$ for the set of ends $\omega'$ of $G$ such that $C(X,\omega') = C$. The collection of all sets $\Omega(X,\omega)$, with $X \subseteq V(G)$ finite and $\omega$ an arbitrary end, forms a basis for the usual topology on $\Omega(G)$.

In this topology, a net $\varphi \in \textsc{Nets}(\Omega(G))$ converges to an end $\omega$ if for every finite set $F \subseteq V(G)$ there exists $d \in \mathrm{dom}(\varphi)$ such that $\varphi[d^\uparrow] \subseteq \Omega(F,\omega)$.

We now define a new convergence on $\Omega(G)$, denoted by $\to_\Omega$. For an end $\omega \in \Omega(G)$ and $k \in \mathbb{N}$, let $\Omega(k,\omega)$ denote the set of ends that admit a representative for which there are at least $k$ vertex-disjoint paths to some representative of $\omega$.

We say that a net $\varphi \in \textsc{Nets}(\Omega(G))$ converges to $\omega \in \Omega(G)$ if for every $k \in \mathbb{N}$ there exists $d_k \in \mathrm{dom}(\varphi)$ such that $\varphi[d_k^\uparrow] \subseteq \Omega(k,\omega)$. In other words, a net of ends converges if, from some point onward, it is possible to establish an arbitrarily large number of vertex-disjoint connections to $\omega$.

Notice that this convergence is strictly stronger than the usual one (see Figure~\ref{exm4}). It also defines a topology, which remains zero-dimensional and, unlike the usual topology, is first countable. In this setting, local finiteness of the graph no longer implies compactness, as it does in the usual topology. This raises a number of questions: Which known results from the literature still hold under this new convergence, and what new results can be obtained? Moreover, can existing results in the literature be proved more simply using nets?

\begin{figure}[H]
	\centering

	\tikzset{every picture/.style={line width=0.75pt}} 
	
	\begin{tikzpicture}[x=0.75pt,y=0.75pt,yscale=-1,xscale=1]
		
		\draw    (100,191.2) -- (120,191.2) ;
		\draw [shift={(120,191.2)}, rotate = 0] [color={rgb, 255:red, 0; green, 0; blue, 0 }  ][fill={rgb, 255:red, 0; green, 0; blue, 0 }  ][line width=0.75]      (0, 0) circle [x radius= 3.35, y radius= 3.35]   ;
		\draw [shift={(100,191.2)}, rotate = 0] [color={rgb, 255:red, 0; green, 0; blue, 0 }  ][fill={rgb, 255:red, 0; green, 0; blue, 0 }  ][line width=0.75]      (0, 0) circle [x radius= 3.35, y radius= 3.35]   ;
		\draw    (120,191.2) -- (140,191.2) ;
		\draw [shift={(140,191.2)}, rotate = 0] [color={rgb, 255:red, 0; green, 0; blue, 0 }  ][fill={rgb, 255:red, 0; green, 0; blue, 0 }  ][line width=0.75]      (0, 0) circle [x radius= 3.35, y radius= 3.35]   ;
		\draw [shift={(120,191.2)}, rotate = 0] [color={rgb, 255:red, 0; green, 0; blue, 0 }  ][fill={rgb, 255:red, 0; green, 0; blue, 0 }  ][line width=0.75]      (0, 0) circle [x radius= 3.35, y radius= 3.35]   ;
		\draw    (140,191.2) -- (160,191.2) ;
		\draw [shift={(160,191.2)}, rotate = 0] [color={rgb, 255:red, 0; green, 0; blue, 0 }  ][fill={rgb, 255:red, 0; green, 0; blue, 0 }  ][line width=0.75]      (0, 0) circle [x radius= 3.35, y radius= 3.35]   ;
		\draw [shift={(140,191.2)}, rotate = 0] [color={rgb, 255:red, 0; green, 0; blue, 0 }  ][fill={rgb, 255:red, 0; green, 0; blue, 0 }  ][line width=0.75]      (0, 0) circle [x radius= 3.35, y radius= 3.35]   ;
		\draw    (160,191.2) -- (180,191.2) ;
		\draw [shift={(180,191.2)}, rotate = 0] [color={rgb, 255:red, 0; green, 0; blue, 0 }  ][fill={rgb, 255:red, 0; green, 0; blue, 0 }  ][line width=0.75]      (0, 0) circle [x radius= 3.35, y radius= 3.35]   ;
		\draw [shift={(160,191.2)}, rotate = 0] [color={rgb, 255:red, 0; green, 0; blue, 0 }  ][fill={rgb, 255:red, 0; green, 0; blue, 0 }  ][line width=0.75]      (0, 0) circle [x radius= 3.35, y radius= 3.35]   ;
		\draw    (180,191.2) -- (200,191.2) ;
		\draw [shift={(200,191.2)}, rotate = 0] [color={rgb, 255:red, 0; green, 0; blue, 0 }  ][fill={rgb, 255:red, 0; green, 0; blue, 0 }  ][line width=0.75]      (0, 0) circle [x radius= 3.35, y radius= 3.35]   ;
		\draw [shift={(180,191.2)}, rotate = 0] [color={rgb, 255:red, 0; green, 0; blue, 0 }  ][fill={rgb, 255:red, 0; green, 0; blue, 0 }  ][line width=0.75]      (0, 0) circle [x radius= 3.35, y radius= 3.35]   ;
		\draw    (100,191.2) -- (100,171.2) ;
		\draw [shift={(100,171.2)}, rotate = 270] [color={rgb, 255:red, 0; green, 0; blue, 0 }  ][fill={rgb, 255:red, 0; green, 0; blue, 0 }  ][line width=0.75]      (0, 0) circle [x radius= 3.35, y radius= 3.35]   ;
		\draw [shift={(100,191.2)}, rotate = 270] [color={rgb, 255:red, 0; green, 0; blue, 0 }  ][fill={rgb, 255:red, 0; green, 0; blue, 0 }  ][line width=0.75]      (0, 0) circle [x radius= 3.35, y radius= 3.35]   ;
		\draw    (100,171.2) -- (100,151.2) ;
		\draw [shift={(100,151.2)}, rotate = 270] [color={rgb, 255:red, 0; green, 0; blue, 0 }  ][fill={rgb, 255:red, 0; green, 0; blue, 0 }  ][line width=0.75]      (0, 0) circle [x radius= 3.35, y radius= 3.35]   ;
		\draw [shift={(100,171.2)}, rotate = 270] [color={rgb, 255:red, 0; green, 0; blue, 0 }  ][fill={rgb, 255:red, 0; green, 0; blue, 0 }  ][line width=0.75]      (0, 0) circle [x radius= 3.35, y radius= 3.35]   ;
		\draw    (100,151.2) -- (100,131.2) ;
		\draw [shift={(100,131.2)}, rotate = 270] [color={rgb, 255:red, 0; green, 0; blue, 0 }  ][fill={rgb, 255:red, 0; green, 0; blue, 0 }  ][line width=0.75]      (0, 0) circle [x radius= 3.35, y radius= 3.35]   ;
		\draw [shift={(100,151.2)}, rotate = 270] [color={rgb, 255:red, 0; green, 0; blue, 0 }  ][fill={rgb, 255:red, 0; green, 0; blue, 0 }  ][line width=0.75]      (0, 0) circle [x radius= 3.35, y radius= 3.35]   ;
		\draw    (100,131.2) -- (100,111.2) ;
		\draw [shift={(100,111.2)}, rotate = 270] [color={rgb, 255:red, 0; green, 0; blue, 0 }  ][fill={rgb, 255:red, 0; green, 0; blue, 0 }  ][line width=0.75]      (0, 0) circle [x radius= 3.35, y radius= 3.35]   ;
		\draw [shift={(100,131.2)}, rotate = 270] [color={rgb, 255:red, 0; green, 0; blue, 0 }  ][fill={rgb, 255:red, 0; green, 0; blue, 0 }  ][line width=0.75]      (0, 0) circle [x radius= 3.35, y radius= 3.35]   ;
		\draw    (100,111.2) -- (100,91.2) ;
		\draw [shift={(100,91.2)}, rotate = 270] [color={rgb, 255:red, 0; green, 0; blue, 0 }  ][fill={rgb, 255:red, 0; green, 0; blue, 0 }  ][line width=0.75]      (0, 0) circle [x radius= 3.35, y radius= 3.35]   ;
		\draw [shift={(100,111.2)}, rotate = 270] [color={rgb, 255:red, 0; green, 0; blue, 0 }  ][fill={rgb, 255:red, 0; green, 0; blue, 0 }  ][line width=0.75]      (0, 0) circle [x radius= 3.35, y radius= 3.35]   ;
		\draw    (120,191.2) -- (120,171.2) ;
		\draw [shift={(120,171.2)}, rotate = 270] [color={rgb, 255:red, 0; green, 0; blue, 0 }  ][fill={rgb, 255:red, 0; green, 0; blue, 0 }  ][line width=0.75]      (0, 0) circle [x radius= 3.35, y radius= 3.35]   ;
		\draw [shift={(120,191.2)}, rotate = 270] [color={rgb, 255:red, 0; green, 0; blue, 0 }  ][fill={rgb, 255:red, 0; green, 0; blue, 0 }  ][line width=0.75]      (0, 0) circle [x radius= 3.35, y radius= 3.35]   ;
		\draw    (120,171.2) -- (120,151.2) ;
		\draw [shift={(120,151.2)}, rotate = 270] [color={rgb, 255:red, 0; green, 0; blue, 0 }  ][fill={rgb, 255:red, 0; green, 0; blue, 0 }  ][line width=0.75]      (0, 0) circle [x radius= 3.35, y radius= 3.35]   ;
		\draw [shift={(120,171.2)}, rotate = 270] [color={rgb, 255:red, 0; green, 0; blue, 0 }  ][fill={rgb, 255:red, 0; green, 0; blue, 0 }  ][line width=0.75]      (0, 0) circle [x radius= 3.35, y radius= 3.35]   ;
		\draw    (120,151.2) -- (120,131.2) ;
		\draw [shift={(120,131.2)}, rotate = 270] [color={rgb, 255:red, 0; green, 0; blue, 0 }  ][fill={rgb, 255:red, 0; green, 0; blue, 0 }  ][line width=0.75]      (0, 0) circle [x radius= 3.35, y radius= 3.35]   ;
		\draw [shift={(120,151.2)}, rotate = 270] [color={rgb, 255:red, 0; green, 0; blue, 0 }  ][fill={rgb, 255:red, 0; green, 0; blue, 0 }  ][line width=0.75]      (0, 0) circle [x radius= 3.35, y radius= 3.35]   ;
		\draw    (120,131.2) -- (120,111.2) ;
		\draw [shift={(120,111.2)}, rotate = 270] [color={rgb, 255:red, 0; green, 0; blue, 0 }  ][fill={rgb, 255:red, 0; green, 0; blue, 0 }  ][line width=0.75]      (0, 0) circle [x radius= 3.35, y radius= 3.35]   ;
		\draw [shift={(120,131.2)}, rotate = 270] [color={rgb, 255:red, 0; green, 0; blue, 0 }  ][fill={rgb, 255:red, 0; green, 0; blue, 0 }  ][line width=0.75]      (0, 0) circle [x radius= 3.35, y radius= 3.35]   ;
		\draw    (120,111.2) -- (120,91.2) ;
		\draw [shift={(120,91.2)}, rotate = 270] [color={rgb, 255:red, 0; green, 0; blue, 0 }  ][fill={rgb, 255:red, 0; green, 0; blue, 0 }  ][line width=0.75]      (0, 0) circle [x radius= 3.35, y radius= 3.35]   ;
		\draw [shift={(120,111.2)}, rotate = 270] [color={rgb, 255:red, 0; green, 0; blue, 0 }  ][fill={rgb, 255:red, 0; green, 0; blue, 0 }  ][line width=0.75]      (0, 0) circle [x radius= 3.35, y radius= 3.35]   ;
		\draw    (140,191.2) -- (140,171.2) ;
		\draw [shift={(140,171.2)}, rotate = 270] [color={rgb, 255:red, 0; green, 0; blue, 0 }  ][fill={rgb, 255:red, 0; green, 0; blue, 0 }  ][line width=0.75]      (0, 0) circle [x radius= 3.35, y radius= 3.35]   ;
		\draw [shift={(140,191.2)}, rotate = 270] [color={rgb, 255:red, 0; green, 0; blue, 0 }  ][fill={rgb, 255:red, 0; green, 0; blue, 0 }  ][line width=0.75]      (0, 0) circle [x radius= 3.35, y radius= 3.35]   ;
		\draw    (140,171.2) -- (140,151.2) ;
		\draw [shift={(140,151.2)}, rotate = 270] [color={rgb, 255:red, 0; green, 0; blue, 0 }  ][fill={rgb, 255:red, 0; green, 0; blue, 0 }  ][line width=0.75]      (0, 0) circle [x radius= 3.35, y radius= 3.35]   ;
		\draw [shift={(140,171.2)}, rotate = 270] [color={rgb, 255:red, 0; green, 0; blue, 0 }  ][fill={rgb, 255:red, 0; green, 0; blue, 0 }  ][line width=0.75]      (0, 0) circle [x radius= 3.35, y radius= 3.35]   ;
		\draw    (140,151.2) -- (140,131.2) ;
		\draw [shift={(140,131.2)}, rotate = 270] [color={rgb, 255:red, 0; green, 0; blue, 0 }  ][fill={rgb, 255:red, 0; green, 0; blue, 0 }  ][line width=0.75]      (0, 0) circle [x radius= 3.35, y radius= 3.35]   ;
		\draw [shift={(140,151.2)}, rotate = 270] [color={rgb, 255:red, 0; green, 0; blue, 0 }  ][fill={rgb, 255:red, 0; green, 0; blue, 0 }  ][line width=0.75]      (0, 0) circle [x radius= 3.35, y radius= 3.35]   ;
		\draw    (140,131.2) -- (140,111.2) ;
		\draw [shift={(140,111.2)}, rotate = 270] [color={rgb, 255:red, 0; green, 0; blue, 0 }  ][fill={rgb, 255:red, 0; green, 0; blue, 0 }  ][line width=0.75]      (0, 0) circle [x radius= 3.35, y radius= 3.35]   ;
		\draw [shift={(140,131.2)}, rotate = 270] [color={rgb, 255:red, 0; green, 0; blue, 0 }  ][fill={rgb, 255:red, 0; green, 0; blue, 0 }  ][line width=0.75]      (0, 0) circle [x radius= 3.35, y radius= 3.35]   ;
		\draw    (140,111.2) -- (140,91.2) ;
		\draw [shift={(140,91.2)}, rotate = 270] [color={rgb, 255:red, 0; green, 0; blue, 0 }  ][fill={rgb, 255:red, 0; green, 0; blue, 0 }  ][line width=0.75]      (0, 0) circle [x radius= 3.35, y radius= 3.35]   ;
		\draw [shift={(140,111.2)}, rotate = 270] [color={rgb, 255:red, 0; green, 0; blue, 0 }  ][fill={rgb, 255:red, 0; green, 0; blue, 0 }  ][line width=0.75]      (0, 0) circle [x radius= 3.35, y radius= 3.35]   ;
		\draw    (160,191.2) -- (160,171.2) ;
		\draw [shift={(160,171.2)}, rotate = 270] [color={rgb, 255:red, 0; green, 0; blue, 0 }  ][fill={rgb, 255:red, 0; green, 0; blue, 0 }  ][line width=0.75]      (0, 0) circle [x radius= 3.35, y radius= 3.35]   ;
		\draw [shift={(160,191.2)}, rotate = 270] [color={rgb, 255:red, 0; green, 0; blue, 0 }  ][fill={rgb, 255:red, 0; green, 0; blue, 0 }  ][line width=0.75]      (0, 0) circle [x radius= 3.35, y radius= 3.35]   ;
		\draw    (180,191.2) -- (180,171.2) ;
		\draw [shift={(180,171.2)}, rotate = 270] [color={rgb, 255:red, 0; green, 0; blue, 0 }  ][fill={rgb, 255:red, 0; green, 0; blue, 0 }  ][line width=0.75]      (0, 0) circle [x radius= 3.35, y radius= 3.35]   ;
		\draw [shift={(180,191.2)}, rotate = 270] [color={rgb, 255:red, 0; green, 0; blue, 0 }  ][fill={rgb, 255:red, 0; green, 0; blue, 0 }  ][line width=0.75]      (0, 0) circle [x radius= 3.35, y radius= 3.35]   ;
		\draw    (160,171.2) -- (160,151.2) ;
		\draw [shift={(160,151.2)}, rotate = 270] [color={rgb, 255:red, 0; green, 0; blue, 0 }  ][fill={rgb, 255:red, 0; green, 0; blue, 0 }  ][line width=0.75]      (0, 0) circle [x radius= 3.35, y radius= 3.35]   ;
		\draw [shift={(160,171.2)}, rotate = 270] [color={rgb, 255:red, 0; green, 0; blue, 0 }  ][fill={rgb, 255:red, 0; green, 0; blue, 0 }  ][line width=0.75]      (0, 0) circle [x radius= 3.35, y radius= 3.35]   ;
		\draw    (160,151.2) -- (160,131.2) ;
		\draw [shift={(160,131.2)}, rotate = 270] [color={rgb, 255:red, 0; green, 0; blue, 0 }  ][fill={rgb, 255:red, 0; green, 0; blue, 0 }  ][line width=0.75]      (0, 0) circle [x radius= 3.35, y radius= 3.35]   ;
		\draw [shift={(160,151.2)}, rotate = 270] [color={rgb, 255:red, 0; green, 0; blue, 0 }  ][fill={rgb, 255:red, 0; green, 0; blue, 0 }  ][line width=0.75]      (0, 0) circle [x radius= 3.35, y radius= 3.35]   ;
		\draw    (160,131.2) -- (160,111.2) ;
		\draw [shift={(160,111.2)}, rotate = 270] [color={rgb, 255:red, 0; green, 0; blue, 0 }  ][fill={rgb, 255:red, 0; green, 0; blue, 0 }  ][line width=0.75]      (0, 0) circle [x radius= 3.35, y radius= 3.35]   ;
		\draw [shift={(160,131.2)}, rotate = 270] [color={rgb, 255:red, 0; green, 0; blue, 0 }  ][fill={rgb, 255:red, 0; green, 0; blue, 0 }  ][line width=0.75]      (0, 0) circle [x radius= 3.35, y radius= 3.35]   ;
		\draw    (160,111.2) -- (160,91.2) ;
		\draw [shift={(160,91.2)}, rotate = 270] [color={rgb, 255:red, 0; green, 0; blue, 0 }  ][fill={rgb, 255:red, 0; green, 0; blue, 0 }  ][line width=0.75]      (0, 0) circle [x radius= 3.35, y radius= 3.35]   ;
		\draw [shift={(160,111.2)}, rotate = 270] [color={rgb, 255:red, 0; green, 0; blue, 0 }  ][fill={rgb, 255:red, 0; green, 0; blue, 0 }  ][line width=0.75]      (0, 0) circle [x radius= 3.35, y radius= 3.35]   ;
		\draw    (180,171.2) -- (180,151.2) ;
		\draw [shift={(180,151.2)}, rotate = 270] [color={rgb, 255:red, 0; green, 0; blue, 0 }  ][fill={rgb, 255:red, 0; green, 0; blue, 0 }  ][line width=0.75]      (0, 0) circle [x radius= 3.35, y radius= 3.35]   ;
		\draw [shift={(180,171.2)}, rotate = 270] [color={rgb, 255:red, 0; green, 0; blue, 0 }  ][fill={rgb, 255:red, 0; green, 0; blue, 0 }  ][line width=0.75]      (0, 0) circle [x radius= 3.35, y radius= 3.35]   ;
		\draw    (180,151.2) -- (180,131.2) ;
		\draw [shift={(180,131.2)}, rotate = 270] [color={rgb, 255:red, 0; green, 0; blue, 0 }  ][fill={rgb, 255:red, 0; green, 0; blue, 0 }  ][line width=0.75]      (0, 0) circle [x radius= 3.35, y radius= 3.35]   ;
		\draw [shift={(180,151.2)}, rotate = 270] [color={rgb, 255:red, 0; green, 0; blue, 0 }  ][fill={rgb, 255:red, 0; green, 0; blue, 0 }  ][line width=0.75]      (0, 0) circle [x radius= 3.35, y radius= 3.35]   ;
		\draw    (180,131.2) -- (180,111.2) ;
		\draw [shift={(180,111.2)}, rotate = 270] [color={rgb, 255:red, 0; green, 0; blue, 0 }  ][fill={rgb, 255:red, 0; green, 0; blue, 0 }  ][line width=0.75]      (0, 0) circle [x radius= 3.35, y radius= 3.35]   ;
		\draw [shift={(180,131.2)}, rotate = 270] [color={rgb, 255:red, 0; green, 0; blue, 0 }  ][fill={rgb, 255:red, 0; green, 0; blue, 0 }  ][line width=0.75]      (0, 0) circle [x radius= 3.35, y radius= 3.35]   ;
		\draw    (180,111.2) -- (180,91.2) ;
		\draw [shift={(180,91.2)}, rotate = 270] [color={rgb, 255:red, 0; green, 0; blue, 0 }  ][fill={rgb, 255:red, 0; green, 0; blue, 0 }  ][line width=0.75]      (0, 0) circle [x radius= 3.35, y radius= 3.35]   ;
		\draw [shift={(180,111.2)}, rotate = 270] [color={rgb, 255:red, 0; green, 0; blue, 0 }  ][fill={rgb, 255:red, 0; green, 0; blue, 0 }  ][line width=0.75]      (0, 0) circle [x radius= 3.35, y radius= 3.35]   ;
		\draw  [dash pattern={on 0.84pt off 2.51pt}]  (100,69.6) -- (100,82.6) ;
		\draw  [dash pattern={on 0.84pt off 2.51pt}]  (120,70.6) -- (120,83.6) ;
		\draw  [dash pattern={on 0.84pt off 2.51pt}]  (139,70.6) -- (139,83.6) ;
		\draw  [dash pattern={on 0.84pt off 2.51pt}]  (160,71.6) -- (160,84.6) ;
		\draw  [dash pattern={on 0.84pt off 2.51pt}]  (181,72.6) -- (181,85.6) ;
		\draw  [dash pattern={on 0.84pt off 2.51pt}]  (220,191.6) -- (210,191.6) ;
		\draw  [dash pattern={on 0.84pt off 2.51pt}]  (210,142.6) -- (200,142.6) ;

	\end{tikzpicture}
	\caption{The sequence of vertical rays converges to the horizontal ray with respect to the usual topology on the end space, but not with respect to $\to_{\Omega}$.}
	\label{exm4}
\end{figure}

\section*{Acknowledgments}
The third named author thanks the support of Coordenação de Aperfeiçoamento de Pessoal de Nível Superior (CAPES), being sponsored through grant number 88887.136056/2025-00. We thank Guilherme Eduardo Pinto for the nice discussions and helpful suggestions.  Finally,
we thank the referee of an earlier submission for their valuable comments, which
helped improve the exposition of this work.

\bibliographystyle{amsplain}

\bibliography{refs}

\end{document}